\documentclass{amsart}
\usepackage{amssymb,amsmath,amsthm,epsf,epsfig,dsfont,bbm}
\usepackage[margin=90pt]{geometry}
\usepackage{verbatim}
\usepackage{latexsym}
\usepackage[utf8]{inputenc}
\usepackage[backref=page]{hyperref}
\hypersetup{
	colorlinks   = true,
	citecolor    = blue,
	linkcolor    = red 
}
\usepackage{upref, eucal}
\usepackage[all]{xy}

\usepackage{autonum}

\newcommand {\nc} {\newcommand}
\newcommand {\enm} {\ensuremath}

\def \d{\delta}

\nc {\bdm} {\begin{displaymath}}
\nc {\edm} {\end{displaymath}}

\newtheorem {theorem} {\bf{Theorem}}[section]
\newtheorem {lemma}[theorem] {\bf Lemma}
\newtheorem {proposition}[theorem] {\bf Proposition}

\newtheorem {corollary}[theorem] {\bf Corollary}
\numberwithin {equation}{section}

\newcommand\BB{\mathbb{B}}\newcommand\FF{\mathbb{F}}\newcommand\QQ{\mathbb{Q}}\newcommand\WW{\mathbb{W}}
\newcommand\ZZ{\mathbb{Z}}

\newcommand{\Ou}{\enm{\mathcal{O}}}

\nc{\J}{\enm{\mathcal{J} }}
\nc {\Z} {\enm{\mathbb{Z}}}
\nc {\form}[1] {\enm{\mbox{\underline{for}}}_{#1}}
\nc {\prol}[1] {\enm{\mbox{\underline{prol}}_{{#1}^*}}}

\nc {\stk} {\stackrel}

\newcommand{\map}{\rightarrow}

\newcommand{\inj}{\hookrightarrow}

\newcommand{\Pn}[2] {\ensuremath{ {\mathbb{P}}^{#1}_{#2}}}
\nc{\Quot}[3]{\enm{ {\mathfrak{Quot}_{ {#1}/{#2}/{#3}}}}}
\nc{\Hilb}[2]{\enm{ {\mathfrak{Hilb}_{ {#1}/{#2}}}}}
\newcommand{\mfrak}[1]{\mathfrak{#1}}

\newcommand{\bb}[1]{\mathbb{#1}}
\newcommand{\mcal}[1]{\mathcal{#1}}

\nc {\Coh}[4] {\ensuremath{H^{#1}(\Pn{#2}{},{#3}({#4}))}}
\nc {\Ch}[3] {\enm{H^{#1}(X_t,{#2}_t({#3}))}}
\nc {\Qphi}[4]{\enm{ {\mathfrak{Quot}^{~#4}_{ {#1}/{#2}/{#3}}}}}
\nc {\Gra}[4]{\enm{ {\mathfrak{Grass}_{#2}({#3},{#4})}}}
\nc {\HomA}[2]{\enm{\mathrm{Hom}_A{#1}{#2}}}
\nc {\tr}{\mathrm{tr}}

\newcommand{\ov}[1]{\overline{#1}}

\nc {\C}[2]{\enm{\left(\begin{array}{l} {#1} \\ {#2} \end{array} \right)}}
\nc {\mat}[4]{\enm{\left(\begin{array}{ll}{#1} & {#2} \\ {#3} & {#4}
\end{array}\right)}}

\def \vp{\varphi}
\def \mb{\mbox}

 \def \Z{{\mathbb Z}}

   \def \h{\hat{\ }}

\def \d{\delta} \def \bZ{{\mathbb Z}}  \def \bF{{\bf F}}

  \def \bX{{\bf X}} \def \bH{{\bf H}}
   \def \bF{{\mathbb F}}

\def \hG{\hat{\mathbb{G}}_{\mathrm{a}}}
\def \hGm{\hat{\mathbb{G}}_{\mathrm{m}}}

\def \hA{\hat{A}}

\def \R1{R((q))[q']\h}

\usepackage{xcolor}

\newcommand{\lam}{\lambda}
\DeclareMathOperator{\Spec}{\mathrm{Spec}}
\DeclareMathOperator{\Spf}{\mathrm{Spf}}
\DeclareMathOperator{\Lie}{\mathrm{Lie}}
\DeclareMathOperator{\rk}{\mathrm{rk}}
\newcommand{\Hom}{\mathrm{Hom}}
\newcommand{\End}{\mathrm{End}}
\newcommand{\Ext}{\mathrm{Ext}}

\newcommand{\bI}{{\bf I}}
\newcommand{\switt}{s_{\mathrm{Witt}}}

\newcommand{\longlabelmap}[1]{{\,\buildrel #1\over\longrightarrow\,}}
\newcommand{\longmap}{{\,\longrightarrow\,}}
\newcommand{\beqar}{\begin{eqnarray*}}
\newcommand{\eeqar}{\end{eqnarray*}}

\nc{\bx}{\mathbf{x}}
\nc{\by}{\mathbf{y}}
\nc{\bz}{\mathbf{z}}
\nc{\ba}{\mathbf{a}}
\nc{\Fp}{\tilde{F}}
\nc{\Rp}{\tilde{R}}
\nc{\mlow}{m_{\mathrm{l}}}
\nc{\mup}{m_{\mathrm{u}}}
\nc{\ord}{\mb{ord }}
\nc{\bXp}{\bX_{\mathrm{prim}}}
\nc{\bPsi}{\mathbf{\Psi}}
\nc{\mult}{\mathrm{mult}}
\nc{\mbB}{\mathbbm{B}}
\nc{\mfor}[1]{{#1}^{\mathrm{for}}}
\nc{\Hdr}{\bH^1_{\mathrm{dR}}}
\nc{\Hcr}{\bH^1_{\mathrm{cris}}}
\nc{\Fc}{F_{\mathrm{cris}}}
\nc {\Hd}{\bH_{\d}}
\nc{\mn}{[m]n}

\nc{\tF}{\tilde{F}}
\nc{\fra}{\mfrak{f}}
\nc{\bt}{{\bf t}}
\nc{\Nn}{N^{\mn}}
\nc{\del}{\Delta}
\nc{\tilW}{\tilde{W}}
\nc{\bo}{{\bf b}}

\nc{\Di}[2]{\Delta^{#1}i^*\d^{#2}}
\nc{\di}[1]{i^*\d^{#1}}
\newcommand{\oG}{\overline{G}}
\nc{\pr}{\mathrm{pr}}
\nc{\Mat}{\mathrm{Mat}}

\nc{\tTheta}{\tilde{\Theta}}
\newcommand{\mfg}{\mathfrak{g}}
\nc{\clift}{\beta}
\nc{\Gfor}{G^{\mathrm{for}}}
\nc{\cblue}{\color{black}}
\nc{\cred}{\color{purple}}
\nc{\Fcris}{F_{\mathrm{cris}}}

\newcommand{\cb}[1]{\color{black} {#1} \color{black}}

\newcommand{\mr}[1]{\mathrm{#1}}
\newcommand{\Iso}{FIso \hspace{.05cm}}

\nc{\hGa}{\bb{G}_a^{\mathrm{for}}}
\newcommand{\ainj}[1]{\ar@{^{(}->}[{#1}]}

\DeclareRobustCommand{\loongleftarrow}{%
\leftarrow\joinrel\relbar\joinrel\relbar\DOTSB
}

\title
[Delta Characters and Filtered Isocrystals]
{Delta Characters and Filtered Isocrystals}

\author{Lance Gurney}
\address{Department of Mathematics and Statistics,
University of Melbourne, Victoria 3010 Australia}
\email{lance.gurney@unimelb.edu.au}

\author{Sudip Pandit}
\address{Department of Mathematics, 
King's College London, Strand, London WC2R 2LS, UK}
\email{sudip.pandit@kcl.ac.uk}

\author{Arnab Saha}
\address{Department of Mathematics, 
Indian Institute of Technology Gandhinagar, Gujarat 382355, India}
\email{arnab.saha@iitgn.ac.in}

\date{}



\begin{document}
\maketitle

\begin{abstract}
\cblue

Given an abelian scheme $A$ over a $p$-adic ring $R$, Borger and Saha constructed a filtered module $\{\mathbf{H}_\delta(A) \supset \mathbf{X}_{\mathrm{prim}}(A)\supset \{0\}\}$ with a semilinear operator $\mathfrak{f}^*$ on $\mathbf{H}_\delta(A)$ using the theory of arithmetic jet spaces.
The above object admits a canonical map $\Phi$ to the 
Hodge sequence $\{\mathbf{H}^1_{\mathrm{dR}}(A) \supset H^0(A,\Omega_A)\supset \{0\}\}$ of $A$
in the category of filtered modules. As a result, by restricting $\Phi$, we
obtain a natural $R$-linear map $\Upsilon: \mathbf{X}_{\mathrm{prim}}(A) \rightarrow 
H^0(A,\Omega_A)$.

In this paper, we show that the map $\Upsilon$ is an isomorphism of vector
spaces over $K$, the field of fractions of $R$. As a consequence, we will show that for all abelian schemes $A$, the operator $\mathfrak{f}^*$ on $\mathbf{H}_\delta(A)_K$ 
is a bijection, and
our object $\{\mathbf{H}_\delta(A)_K \supset
\mathbf{X}_{\mathrm{prim}}(A)_K\supset \{0\}\}$ becomes
 a filtered isocrystal. 
In fact, the above results admit a generalization to the setting of semi-abelian schemes.
The elements of $\mathbf{X}_{\mathrm{prim}}(A)$ are represented by 
primitive additive 
characters of the arithmetic jet spaces attached to $A$. 
Hence, our isomorphism given by $\Upsilon$ provides an interesting 
character-theoretic interpretation of $H^0(A,\Omega_A)$  
in terms of primitive delta characters. As a result, to any $1$-form $\omega$, the above isomorphism associates a canonical 
numerical invariant that depends on deformation theoretic data of $A$.

Furthermore, we also extend a
comparison theorem between $\mathbf{H}_\delta(A)_K$ 
and the first crystalline cohomology 
$\mathbf{H}_{\mathrm{cris}}^1(A)_K$ in the general case when the elliptic curve $A$ is defined over 
the ring of integers of a $p$-adic field $K$ that is a finite extension of
$\mathbb{Q}_p$.

%Using our results, the second author showed that the original map $\Phi$ is in fact a map of $F$-isocrystals which would imply that 
%$\mathbf{H}_\delta(A)_K$ is canonically isomorphic to the fundamental smallest 
%subisocrystal of $\mathbf{H}^1_{\mathrm{cris}}(A_0)_K$ containing $H^0(A,\Omega_A)$.

\end{abstract}

\color{black}

\section{Introduction}
Drawing inspiration from differential algebraic geometry, Buium initiated the theory of arithmetic jet spaces in the category of $p$-adic formal schemes \cite{bui95} and studied the additive characters of the arithmetic jet spaces of abelian schemes over a $p$-adic field, which we call delta characters of the abelian scheme. The arithmetic jets and delta characters encode subtle arithmetic information about abelian varieties. For example, the torsion points of an abelian variety have been studied using delta characters and their kernels, which remarkably led to Buium's proof of the effective Manin-Mumford conjecture for curves \cite{bui96},  and very recently, explicit bounds on the Mordell-Lang  and Zilber--Pink locus for curves by Dogra and Pandit \cite{DP25, DP26} and a higher-dimensional unramified Manin-Mumford bound by Miller-Morrow using $p$-adic methods in \cite{MM25}. The theory of arithmetic jet spaces has been further developed in many articles, such as \cite{bui00, barc, busa1, busa2, hurl}, from the perspective of modular forms, and has been applied to prove finiteness results in modular--elliptic correspondences by Buium and Poonen in \cite{BP09}. A detailed exposition of the subject can be found in \cite{bui-book}, and more recent developments are covered in \cite{BL, buium-miller-0, buium-miller}.

In \cite{BS_a}, associated to an abelian scheme $A$,
Borger and Saha constructed a canonical filtered vector space 
$\{\Hd(A) \supset \bXp(A)\}$ with a semilinear operator $\fra^*$ on it. 
The filtered vector space admits a natural map $\Phi :\bH_\delta(A) \map 
\Hcr(A)$ respecting the Hodge filtration. As a result, the above map $\Phi$ by restriction induces
a canonical map $\Upsilon : \bXp(A) \map H^0(A,\Omega_A)$. In 
\cite{PS-2}, Pandit and Saha showed that in the case when $A$ is an elliptic 
curve over $\ZZ_p$, $\Hd(A)$ is indeed a filtered isocrystal which admitted an 
appropriate
 comparison isomorphism with $\Hcr(A)$. The analogous theory over function
fields was developed in \cite{BS_a} and \cite{PS-1}.

One of the main results
of this paper is to show that the above map $\Upsilon$ is an isomorphism. 
As a consequence, we show that for all abelian schemes $A$, our object
$\{\bH_\delta(A) \supset \bXp(A)\}$ is a filtered isocrystal.
The elements of $\bXp(A)$ are represented by primitive delta 
characters of the arithmetic jet spaces attached to $A$. 
Hence our isomorphism given by $\Upsilon$ gives a very interesting 
character-theoretic
interpretation of $H^0(A,\Omega_A)$ in terms of delta characters.

Using our above results, the second author will show in a 
subsequent paper \cite{Pandit26} that the original map $\Phi:\bH_\delta(A) \map \Hcr(A)$ is 
indeed a morphism of filtered isocrystals-- in other words, $\Phi$ is compatible
with the operator $\fra^*$ on $\bH_\delta(A)$ and the crystalline operator 
$\Fcris$ on
$\Hcr(A)$. This will imply that $\Hd(A)$ is canonically isomorphic to the smallest
subisocrystal of $\Hcr(A)$ containing $H^0(A,\Omega_A)$. Note that the $R$-submodule
$H^0(A,\Omega_A)$ is not always a subisocrystal of $\Hcr(A)$, that is the 
crystalline operator $\Fcris$ does not necessarily preserve $H^0(A,\Omega_A)$.
In fact by Berthelot (Theorem $3.17$ in \cite{BO}), $H^0(A,\Omega_A)$ is a subisocrystal of
$\Hcr(A)$ if and only if there exists an endomorphism of $A$ which lifts the
absolute Frobenius on the special fiber.

Some of the above mentioned results in this paper hold in higher generality of 
semi-abelian schemes which we will now explain in greater detail.

%\color{black}
\subsection{Notations and geometric setup.}
Let $\Ou$ be a Dedekind domain of characteristic $0$ with finite residue
fields and fix a non-zero prime ideal $\mfrak{p}$ in it. Let $k$ be the 
residue field at $\mfrak{p}$ with cardinality $q$ which is a power of some 
prime $p$ and $\pi$ be a uniformizer of $\mfrak{p}\Ou_p$, that is, the image of
$\pi$ in $\Ou_{\mfrak{p}}$ generates its maximal ideal. 
Let $R = W(k)$ be the infinite length $\pi$-typical Witt vectors over $k$. Then the
identity map $\phi= \mathbbm{1}$ is a lift of the $q$-power Frobenius on $k$ 
(which is also the identity on $k$).
Let $K$ be the fraction field of $R$ and 
 $S = \Spf R$ be the $\pi$-formal scheme associated to $R$.

%Let $R$ be a flat
%$\Ou$ algebra which is a discrete valuation ring and is $\pi$-adically, 
%complete, equipped with an endomorphism $\phi: R \map R$ which lifts
%the $q$-power Frobenius mod $\pi$. 
%As for example, $(\Ou,\mfrak{p}) =
%(\Z,p)$ and $R$ can be taken to be the ring of integers of a $p$-dic 
%field $K$ which is a  finite extension of $\Q_p$ with $\phi$ a fixed 
%lift of Frobenius on $R$. 

For any $\pi$-formal scheme $X$ over $S$, one defines the $n$-th arithmetic
jet functor as 
$$
J^nX (B) := X(W_n(B))
$$
where for any $\pi$-adically complete $R$-algebra $B$, $W_n(B)$ is the ring of
$\pi$-typical Witt vectors of length $n+1$ \cite{bor11a, drin76, joyal}.
By Section $(12.8)$ of \cite{bor11b} and Theorem $1.3$ of \cite{bps}, 
the functor $J^nX$ is representable
by a $\pi$-formal $S$-scheme, which we will continue to denote as $J^nX$. 
This is precisely the arithmetic jet space constructed by Buium in 
\cite{bui95}.

The $\pi$-typical Witt vectors admit two natural ring maps: the truncation map $T: 
W_n(B) \map W_{n-1}(B)$ given by 
$$
T(a_0,\dots, a_n) = (a_0,\dots, a_{n-1}),
$$
and the Frobenius map $F : W_n(B) \map W_{n-1}(B)$ given by 
$$
F(a_0,\dots, a_n) = (a_0^q + \pi a_1,\dots ),
$$
for all $(a_0,\dots, a_n) \in W_n(B)$.
The above two ring maps induce morphisms $u$ (projection map)
and $\phi$ (Frobenius map) respectively from
$J^nX$ to $J^{n-1}X$. The system $J^*X =\{J^nX\}_{n=0}^\infty$ is known
as the {\it canonical prolongation sequence}. Note that if $G$ is a 
group object, then $J^nG$ is also naturally a group object by its very definition.

\color{black}
If $G$ is a smooth $\pi$-formal group scheme over $S$, the projection map
$u: J^nG \map G$ is a surjection of $\pi$-formal group schemes 
\cite{bor11b, bui00}. Let $N^nG$ be the kernel of the morphism $u$ and
hence we have the following canonical short exact sequence of $\pi$-formal
group schemes
\begin{align}\label{short1}
0 \map N^nG \stk{\iota}{\map} J^nG \stk{u}{\map} G \map 0.
\end{align} 
Combining Theorem $4.3$ in  \cite{BS_b} and Theorem $1.2$ in \cite{PS-2}, 
one shows that for all $n$ we have a canonical isomorphism
\begin{align}
N^nG \simeq J^{n-1}(N^1G),
\end{align}
and the associated Frobenius morphism for the canonical prolongation 
sequence $\{N^*G\}_{n=0}^\infty$, denoted $\fra: N^nG \map N^{n-1}G$ 
satisfies
\begin{align}
\label{phi-fra}
\phi^{\circ 2} \circ \iota = \phi \circ \iota \circ \fra.
\end{align}
Set $\bX_n(G):= \Hom(J^nG, \hG)$ to be the $R$-module of additive 
characters of the $n$-th arithmetic jet space $J^nG$. The 
$R$-module $\bX_n(G)$ is called the
module of {\it delta characters} of $G$ of order $\leq n$. A delta character
$\Theta$ is said to have exact order $n$ if $\Theta \in \bX_n(G) \backslash 
u^*\bX_{n-1}(G)$ where $u^*$ is the pullback of delta characters by
$u:J^nG \map J^{n-1}G$.
A delta character $\Theta$ is called {\it primitive} if it cannot be written
as a sum of delta characters of lower order via various pullbacks by $\phi^*$.

Let $\bX_\infty(G)$ be the direct limit of $\bX_n(G)$ with respect to 
pulling back by $u^*$. Also pulling back a delta character $\Theta$ by $\phi^*$ 
endows $\bX_\infty(G)$ with a semilinear action of $\phi^*$ satisfying
$\phi^* \cdot r = \phi(r) \cdot \phi^*$ for all $r \in R$, see 
$(3.11)$ in \cite{BS_b}.

We define the $R$-module of {\it primitive} delta characters $\bXp(G)$ as 
$$
\bXp(G) := \frac{\bX_\infty(G)}{\phi^* \bX_\infty(G)}.
$$
By Theorem $8.7$ in \cite{PS-2}, $\bXp(G)$ is a free $R$-module of rank 
$g$ (the result over the fraction field $K$ was proved in \cite{BS_b}).
It can also be shown that $\bXp(G)$ admits an $R$-basis consisting of images of
primitive delta characters.

For each $n$, consider the quotient $R$-module defined by
$$
\bH_n(G) := \frac{\Hom(N^nG,\hG)}{(\phi \circ \iota)^* \bX_{n-1}(G)}.
$$
We define the $R$-module
$$
\bH_\d(G) := \varinjlim_{u^*} \bH_n(G).
$$
The lift of Frobenius $\fra$ induces the operator $\fra^*$ on 
$\bH_\d(G)$ (recalled in Section \ref{pre}) and we have the following short exact 
sequence of $R$-modules
\begin{align}
\label{Hdexact-1}
0 \map \bXp(G) \map \Hd(G) \map \bI(G) \map 0,
\end{align}
where $\bI(G)$ is an $R$-submodule of $\Ext(G,\hG)$ (see \cite{BS_b} and 
\cite{PS-2} for details).

\subsection{Statement of main results.}
A {\it filtered module} over $R$ is an $R$-module $M$ endowed with a system of
decreasing $R$-submodules $M^\bullet =\{M^i\}_{i \in \ZZ}$, that is 
$M^{i+1} \subset M^i$
for all $i \in \ZZ$. The category of such objects will be denoted as 
$\mathrm{Fil}_R$. We define a {\it (non-degenerate) isocrystal} over $K$ to be
a finite dimensional $K$-vector space $D$ equipped with a (bijective) 
Frobenius-semilinear endomorphism $\phi_D: D \map D$. A {\it filtered isocrystal}
over $K$ is a triple $(D,\phi_D, D^\bullet\})$ where $(D,\phi_D)$ is an isocrystal
over $K$ and $(D,D^\bullet\})$ is an object in $\mathrm{Fil}_K$.

Consider the filtered isocrystal 
$$
\mathrm{\Iso}(\bH_\d(G)_{K}) := (\bH_\d(G)_{K}, \fra^*, 
\bH_\d(G)_{K}^\bullet),
$$
where $\bH_\d(G)_{K}^\bullet$ is the filtration given by 
$\bH_\d(G)_{K} \supset \bXp(G)_{K} \supset \{0\}$.
Our first main result is the following:

\begin{theorem}
\label{bijfra-1}
Let $G$ be a semi-abelian $\pi$-formal group scheme of relative dimension $g$ 
over $S$. 
Then the operator $\fra^*: \bH_\d(G)_K \map \bH_\d(G)_K$ is a 
bijection.

Hence $\mathrm{\Iso}(\bH_\d(G)_K)$ is a non-degenerate filtered isocrystal.

\end{theorem}
The above result is proved in Section \ref{non-deg} as Theorem \ref{bijfra}.
When $A$ is a $\pi$-formal abelian scheme over $S$ the short exact sequence 
(\ref{Hdexact-1}) admits a natural map to the Hodge sequence of $A$ 
(recalled in Section \ref{Ext-of-gp-sch} or see ($6.12$)
of \cite{BS_b}) as filtered $R$-modules:
$$\xymatrix{
0 \ar[r] &  \bXp(A) \ar[d]_\Upsilon \ar[r] &  \Hd(A) \ar[d]_\Phi \ar[r] &  
\bI(A) \ainj{d} \ar[r] & 0 \\
0 \ar[r] & H^0(A,\Omega_A) \ar[r] & \Hcr(A) \ar[r] & H^1(A,\Ou_A) \ar[r] & 0. 
}$$

Let $\Upsilon_K: \bXp(A)_K \map H^0(A,\Omega_A)_K$ be the induced 
$K$-linear map obtained upon tensoring with $K$. 

The above map generalizes to
$\Upsilon_K : \bXp(G)_K \map (\Lie G)^*_K$ where $(\Lie G)^* :=
\Hom(\Lie G,\hG)$ for any semi-abelian $\pi$-formal scheme $G$
(described in \eqref{Upsilon-new} of Section \ref{SFrob}).
Our next result is the following:

\begin{theorem}
\label{isoupsemi-1}
Let $G$ be a semi-abelian $\pi$-formal group scheme of relative 
dimension $g$ over $S$. 
Then the  $R$-linear map $\Upsilon: \bXp(G) \map (\Lie G)^*$ is an 
injective morphism of $R$-modules that has $\pi$-power torsion cokernel.

In particular, $\Upsilon_K: \bXp(G)_K \map (\Lie G)^*_K$ is an isomorphism
of $K$-vector spaces.
\end{theorem}

The above result is proved in Section \ref{SUpsilon} as Theorem \ref{isoupsemi}.
It is a natural and intriguing problem to determine the extent to which the 
$R$-module $\bXp(G)$ differs from $(\Lie G)^*$.

{\bf Canonical delta signature of abelian schemes.}
Let $A$ be a $\pi$-formal abelian scheme of relative dimension $g$ over $S$.
An abelian scheme $A$ over $S$ is said to have a {\it canonical lift (CL)}
if there exists a $\beta \in \End(A)$ such that $\beta \mod \pi$ is the
absolute $q$-power Frobenius on $A \otimes_R k$.
It follows from Theorem $7.6$ in \cite{BS_b} that there is an isomorphism 
$\bXp(A) \simeq R\langle \Theta_1,\dots, \Theta_g \rangle$ where for each 
$i=1, \dots ,g$, $\Theta_i$ is a primitive delta character 
of $A$ of exact order $o_i$ and satisfies $1 \leq o_i \leq g+1$.
Hence composing with the above isomorphism in Theorem \ref{isoupsemi-1} we obtain
$$
H^0(A,\Omega_A)_K \simeq \bXp(A)_K \simeq K\langle \Theta_1,\dots, \Theta_g\rangle,
$$
which gives an interesting character theoretic interpretation of $H^0(A,\Omega_A)_K$. 
Moreover note that under the above isomorphism, if 
$\Theta_\omega$ is the delta character associated to an invariant $1$-form $\omega \in
H^0(A,\Omega_A)_K$ then the assignment
$\omega \mapsto o(\Theta_\omega) \in \bb{Z}_{\geq 1}$ where $o(\Theta_\omega)$ is
the exact order of $\Theta$ (which is always $\geq 1$ for an abelian scheme $A$ over $S$), gives a canonical numerical invariant for $\omega$
via the above isomorphism.

Consider the ordered basis set $\mathbb{B} = \{\Theta_1,\dots ,\Theta_g\}$ of 
$\bXp(A)$ with the condition that $o(\Theta_i) \leq o(\Theta_{i+1})$ for all $i=1,
\dots , g-1$.
Consider the association
\begin{align}
A \mapsto \delta\mathrm{-Sign}(A) \in \ZZ_{\geq 1}^g,
\end{align}
where we define the 
{\it $\delta$-signature} of the abelian scheme $A$ over $S$ as
$\delta\mathrm{-Sign}(A) := (o(\Theta_1), \dots, o(\Theta_g))$. 
One can show
that $\delta\mathrm{-Sign}(A)$ is independent of the choice of the ordered
primitive basis $\mathbb{B}$ of $\bXp(A)$. This is a very interesting canonical
numerical invariant which depends on the deformation theoretic data of $A$.

For example, consider the case when $A$ is an elliptic curve over $S$. Then 
by Corollary $7.8$ of \cite{BS_b}, we have
\begin{align}
\delta\mathrm{-Sign}(A):= \left\{\begin{array}{l}
(1), \mbox{if $A$ has CL} \\
(2), \mbox{otherwise}.
\end{array} \right.
\end{align}

Consider the filtered isocrystal of the first crystalline
cohomology 
$$\mathrm{\Iso}(\Hcr(A)_{K}) := (\Hcr(A)_{K}, \Fc,
\Hcr(A)^\bullet_K),$$ 
where $\Hcr(A)^\bullet_K$ is the Hodge filtration $\Hcr(A)_K 
\supset H^0(A,\Omega_A)_K \supset \{0\}$ and $\Fc$ is the crystalline Frobenius
operator on $\Hcr(A)_K$.

Our next theorem is a comparison result between $\bH_\d(A)_K$ and the first
crystalline cohomology of $A$. 
%This can be viewed as a character theoretic
%nterpretation of the first crystalline cohomology of $A$.

\begin{theorem}\label{Iso-crys-11}
Let $A$ be an elliptic curve over $S$. 
%Then $\mathrm{\Iso}(\Hd(A)_{K})$ is
%a weakly admissible object in the category of filtered isocrystals.
\begin{enumerate}
\item If $A$ is a non-CL elliptic curve then
$$
\mathrm{\Iso} (\Hd(A)_{K}) \simeq \mathrm{\Iso} (\Hcr(A)_{K})
$$
in the category of filtered isocrystals.

\item  If $A$ has CL then $$\mathrm{\Iso} (\Hd(A)_{K}) \simeq 
\mathrm{\Iso} (H^0(A, \Omega_A)_{K})$$ 
in the category of filtered isocrystals 
where $\mathrm{\Iso} (H^0(A,\Omega_A)_{K})$
is the one dimensional sub-object of $\mathrm{\Iso}(\Hcr(A)_{K})$.
\end{enumerate}
\end{theorem}

The above result generalizes Theorem $1.7$ in \cite{PS-2} in the case when
$K$ is a finite extension of $\QQ_p$.
Such comparison theorems can be viewed as a character-theoretic interpretation
of the crystalline cohomology.
%Hence it is also a very pertinent question to reconcile our results
%with the theory of prismatic cohomology of Bhatt and Scholze in \cite{bhsch}.

%\cred
%\vspace{.5cm}
%\noindent {\bf Plan of the Paper.} 
%In Section \ref{pre}, we recall the theory of delta characters and the construction of the semilinear object. Then, we recall the explicit relation between the Frobenius on jet spaces and those on their kernels in Section \ref{SFrob}. The difference between these two Frobenius morphisms at the level of characters allows us to describe the $\Upsilon$ map explicitly in a suitable basis. In Section \ref{mod-p}, we establish the compatibility between jet space Frobenius and the Frobenius of the special fiber of a scheme. This allows us to use the non-degeneracy of the Frobenius on the special fiber to derive the invertibility of $\Upsilon$, proving Theorem \ref{isoupsemi-1} in Section \ref{SUpsilon}. In Section \ref{non-deg}, we explicitly describe the matrix of $\mff^*$ with respect to a suitable basis (extending a basis of primitive characters), showing it has a companion-like matrix form composed of small block matrices. Moreover, the top-right corner block turns out to be the product of an invertible matrix and the matrix of $\Upsilon$, which implies the non-degeneracy of $\mff^*$, thereby proving Theorem \ref{bijfra-1}. In the final section, Section \ref{comparison}, we establish the comparison isomorphism between the delta isocrystal and crystalline cohomology for elliptic curves, proving Theorem \ref{Iso-crys-11}.
%\\
\cblue

%\section{Notation}
%\label{notation}
%We collect here some notations fixed throughout the paper.
%\begin{align}
%	p &= \text{a prime number} \\
%\mcal{O}&= \text{a Dedekind domain}\\
%	\mfrak{p}&= \text{a fixed prime ideal of}~ \mcal{O}\\
%	\pi &= \text{a generator of }\mfrak{p} \mathcal{O}_{\mfrak{p}} \\
%	k &= \text{the residue field of $\mcal{O}$ at $\pi$ with cardinality $q$} \\
%	R &= \text{a fixed $\pi$-adically complete $\mcal{O}$-algebra} \\
%	\phi &= \text{an endomorphism of $R$ satisfying $\phi(x) \equiv x^q \bmod \mfrak{p}$, for all $x \in R$} \\
%S &= \Spf R\\
%	M_K &= K\otimes_R M, \text{ for any $R$-module $M$ and } 
%		K=\mathrm{Frac}(R) \\
%	\mfrak{m} &= \text{the maximal ideal of }R\\
%	v_{\pi} &= \text{the valuation on $R$ normalized such that }v_{\pi}(\pi)=1\\
%	e &= \text{the absolute ramification index $v_{\pi}(p)\leq p-2$} \\
%	l &= \text{the residue field of $R$} \\
%G &= \text{a commutative smooth $\pi$-formal group scheme over $\Spf R$}\\
%A &= \text{a $\pi$-formal abelian scheme over $\Spf R$}\\
%a_{q}&= q+1-\#(A(\FF_{q})) \text{ when $A$ is an elliptic curve over 
%$W(\FF_q)$.}\\
%\end{align}
%\cred By a \emph{$\pi$-formal scheme}, we will mean a $\pi$-adic formal scheme over $S$. A semi-abelian scheme $G$ over $R$ is a smooth commutative $\pi$-formal group scheme such that $G$ is an extension of an abelian scheme $A$ by a torus $T$ over $R.$
\cblue

\section{Preliminaries on Delta Geometry}\label{pre}

Let $\Ou$ be a Dedekind domain of characteristic $0$ and $\mfrak{p}$ a 
non-zero prime ideal with
$k$ as the residue field and $q$ be the cardinality of $k$ where $q$ is a
power of a prime $p.$  Let $\pi$ be one of the uniformizers of $\mfrak{p},$ i.e. the image of $\pi$ generates the maximal ideal $\mfrak{p}\Ou_{\mfrak{p}}.$
For any $\Ou$-algebra $B$ and $B$-algebra $A$, we define a $\pi$-derivation 
$\d$ as a set-theoretic map $\d:B \map A$ that satisfies  for 
all $x,y \in B$,

$(i)~ \d(1) = 0$

$(ii)~ \d (x+y) = \d x + \d y + C_\pi(u(x),u(y)) $

$(iii)~ \d(xy) = u(x)^q \d y + u(y)^q \d x + \pi \d x \d y$

where $u:B \map A$ is the structure map and 
$$C_\pi(X,Y) = 
\frac{X^q+Y^q -(X+Y)^q}{\pi}.
$$ 
Given such a $\pi$-derivation $\d$, define $\phi(x):= u(x)^q + \pi\d x$
which is then a ring homomorphism satisfying
$$
\phi(x) \equiv u(x)^q \bmod \mfrak{p}.
$$
We will call such a $\phi$ a {\it lift of Frobenius} with respect to $u$. 
Let $R$ be an $\Ou$-algebra which is a discrete valuation ring and is $\pi$-adically
complete with residue field $l$ and set $S = \Spf R$.

\subsection{$\pi$-typical Witt vectors}
\label{pityp}
We recall some of the basic theory of $\pi$-typical Witt vectors.
The general theory of Witt vectors over Dedekind domains with finite residue
fields were introduced by Borger \cite{bor11a}. The theory over local fields
of characteristic $p$ was introduced earlier by Drinfeld \cite{drin76}.

For an $R$-algebra $B$ with structure map $R \stk{f}{\map} B$, for all 
$n \geq 0$ let $B^{\phi^n}$ be the 
$R$-algebra with the structure map $R \stk{\phi^n}{\map} R \stk{f}{\map} B$.
Given an $S$-scheme $X$,
we define $X^{\phi^n}$ as $X^{\phi^n}(B):= X(B^{\phi^n})$ for
any $R$-algebra $B$. The above functor is
represented by the base change of $X$ over the map $\phi^n:S \map S$ given by
$X^{\phi^n} = X\times_{S,\phi^n} S$. 

We define the {\it ghost rings} $\prod_\phi^n B := B \times B^\phi \times 
\cdots \times B^{\phi^n}$ and $\prod_\phi^\infty B := B \times B^\phi \times 
\cdots $. For all $n \geq 1$, consider the {\it restriction} or 
{\it truncation} map on the ghost rings as $T_w : \prod_\phi^n B \map 
\prod_\phi^{n-1} B$ given by $T_w(w_0,\dots , w_n) := (w_0,\dots, w_{n-1})$.

Also consider the left-shift {\it Frobenius} operators $F_w: \prod_\phi^n B
\map \prod_\phi^{n-1}B$ given by $F_w(w_0,\dots, w_n) = (w_1,\dots, w_n)$. 
Note that $T_w$ is an $R$-algebra map whereas $F_w$ lies over the fixed
Frobenius endomorphism $\phi$ of $R$.

Define as a set $W_n(B) := B^{n+1}$, and the set-theoretic map
$w: W_n(B) \map \prod_\phi^n B$ by $w(x_0,\dots, x_n) := (w_0,\dots ,w_n)$
where for all $i \geq 0$,
$$
w_i = x_0^{{q}^i} + \pi x_1^{{q}^{i-1}}+ \cdots + \pi^i x_i
$$
are the {\it Witt polynomials} and $w$ is known as the {\it ghost} map.
We define the ring of $\pi$-typical Witt vectors of length $n+1$ as the 
following analogous theorem to the $p$-typical case (cf. pp. 171-192 in 
\cite{Mum-66}).

\begin{theorem}
For all $n \geq 0$, there exists a unique functorial $R$-algebra 
structure on $W_n(B)$ such that $w$ is a natural transformation of 
functors of $R$-algebras.
\end{theorem}

We now recall some of the important operators on the $\pi$-typical 
Witt vectors: 

$(1)$ The {\it restriction} or {\it truncation} map given by
$T(x_0,\dots,x_n) = (x_0,\dots, x_{n-1})$ and satisfies $w \circ T
= T_w \circ w$. 

$(2)$ The {\it Frobenius} map $F: W_n(B) \map W_{n-1}(B)$ satisfying
$w \circ F = F_w \circ w$ given by 
$$
F(x_0,\dots ,x_n) = (x_0^{q} + \pi x_1, \dots ).
$$
As in the case of the ghost side map, $F$ lies over the Frobenius
endomorphism $\phi$ of $R$.

Consider the Witt vectors of infinite length 
$W(B):= \varinjlim_T W_n(B)$ be the inverse limit taken over the
truncation map $T$. 

The map $F$ induces the lift of Frobenius map on 
$W(B)$. In fact there exists a unique delta map $\Delta: W(B) \map W(B)$
such that $F$ is the associated Frobenius map, that is 
$F(x) = x^{q} + \pi \Delta (x)$.

The ring $W(B)$ satisfies the following universal property: For any $R$-algebra
$C$ with a $\pi$-derivation $\delta$ on it and an $R$-algebra map $f:C \map B$,
there exists a unique $R$-algebra homomorphism $g: C \map W(B)$ such that 
the diagram 
$$
\xymatrix{
W(B) \ar[d]_T & \\
B & C \ar[l]^f \ar[lu]_g
}
$$
commutes and $g \circ \Delta = g \circ \delta$. Hence $W$ is the right adjoint
of the forgetful functor from $R$-algebras with $\pi$-derivation to $R$-algebra.
Hence via the above universal property of Witt vectors, for any $R$-algebra
$B$, $W(B)$ becomes naturally an $R$-algebra.

For more details, we refer the reader to Section $1$ of \cite{bor11a}. 
This
approach is analogous to that of Joyal in the case of $p$-typical Witt vectors
\cite{joyal}.

\subsection{Prolongation sequences} 
Let $X$ and $Y$ be $\pi$-formal schemes over $S$. We say a pair $(u,\d)$ is a 
{\it prolongation} and we write $Y\stk{(u,\d)}{\longrightarrow} X$, 
if $u:Y \map X$ is a 
map of $\pi$-formal $S$-schemes and $\d: \Ou_X \map u_*\Ou_Y$ is a 
$\pi$-derivation making the following diagram commute:
$$\xymatrix{
R \ar[r] & u_*\Ou_Y \\
R \ar[u]^\d \ar[r] &\Ou_X \ar[u]_\d
}
$$

We now recall the notion of prolongation and arithmetic jet spaces over 
$\pi$-formal schemes as defined by Buium.
%For a more detailed treatment  we refer to \cite{BS_b,bui95}.
As in page 103 in \cite{bui00}, a {\it prolongation 
sequence} is a sequence 
$$
S \stk{(u_0,\d_0)}{\loongleftarrow} T^0 \stk{(u_1,\d_1)}{\loongleftarrow} 
T^1 \stk{(u_2,\d_2)}{\loongleftarrow} \cdots,
$$
where $T^i \stk{(u_{i+1},\d_{i+1})}{\loongleftarrow} T^{i+1}$ are 
prolongations satisfying
$$
u^*_i \circ \d_{i+1} = \d_i \circ u^*_{i+1}
$$
where $u^*_i$ is the pull-back morphism of sheaves induced by $u_i$
for each $i$. We will denote a prolongation sequence as 
$T^*$ or $\{T^n\}_{n\geq 0}$.
Prolongation sequences naturally form a category, where a morphism $f:T^*\to U^*$ is 
a family of morphisms $f^n:T^n\to U^n$ commuting with both the $u$ and $\d$, in the evident sense.

Define $S^*=\{S^i\}_{i=0}^\infty$ to be the prolongation sequence given by 
$S^i := \Spf R$ $u_i := \mathbbm{1}$ and $\d_i = \d$ is the fixed 
$\pi$-derivation on $R$ for all $i$. Let $\mcal{C}_{S^*}$ denote the category
of prolongation sequences over $S^*$.

\cblue 
%Let $\mcal{C}_{S^*}$ denote thecategory of prolongation sequences defined over $S^*$.
For any $\pi$-formal $S$-scheme $X$ and for all $n \geq 0$ we define the 
$n$-th jet space functor $J^nX$ as functor of points to be  
$$
J^nX(B):= X(W_n(B)) = \Hom_S(\Spf (W_n(B)), X)
$$
for any $\pi$-adically complete $R$-algebra $B$. Then $J^nX$ is representable by 
a scheme over $S$ (This was shown in \cite{bor11b} for  schemes over
$S= \Spec \Z$ and in \cite{bps} for a general 
prolongation sequence $S^*$).
The system of $\pi$-formal schemes $J^*X:=\{J^nX\}_{n\geq 0}$ forms a 
prolongation sequence 
and is called the {\it canonical prolongation sequence} as 
in \cite{bui00} where $\phi: J^{n+1}X \map J^nX$ denote the lift of Frobenius
morphism for each $n$. 

By Proposition 1.1 in \cite{bui00}, 
$J^*X$ satisfies the universal property that for any $T^* \in
\mcal{C}_{S^*}$ and $X$ a $\pi$-formal scheme over $S$ we have 
\begin{align}
	\label{univ}
\Hom_S(T^0,X) =\Hom_{\mathcal{C}_{S^*}}(T^*,J^*X).
\end{align}

\subsection{Local Witt coordinates}
\label{localcoordinates}
Consider the $\pi$-formal affine line $H = \hat{\bb{A}}^1 \simeq \Spf R[x]\h$.
For all $\pi$-adically complete $R$-algebra $B$, the adjunction property between
jet space and the Witt functor \cite{bps} \cite{bor11b} we have 
$$
J^nH (B) = \Hom_R(R[x]\h,W_n(B)) \simeq \Hom_R(R[x,x_1,\dots ,x_n]\h, B)
$$
where the above isomorphism is described as follows: for any $f \in 
\Hom_R(R[x]\h,
W_n(B))$ given by $f(x) = (b_0,\dots, b_n)$ where $b_i \in B$ for all $i$, the
corresponding map $\tilde{f} \in \Hom_R(R[x,x_1,\dots, x_n]\h,B)$ is given by
$\tilde{f}(x_i) = b_i$ for all $i$. The coordinate functions $x,x_1,\dots ,x_n$
on $J^nH \simeq \Spf R[x,x_1,\dots,x_n]\h$ are called Witt coordinates. 

%\cred
%\begin{definition}[\'{E}tale coordinate]\label{etale-def}
%Let $G$ be a $\pi$-formal group scheme of relative dimension $g$ over $S$ 
%and $e:S \map G$ be its identity
%section. Then there exists a $\pi$-formal open affine subscheme $U \subset G$ 
%containing the identity section $e$ such that there exists an \'{e}tale morphism
%$U \map H$ where $H= \hat{\bb{A}}^g \simeq \Spf R[\bx]\h,~\bx = \{x_1,\dots x_g
%\}$. This $\bx$ is referred to as a system of \'{e}tale coordinates around the
%identity section of $G$. 
%\end{definition}
\cblue

Suppose $U \map H$ is an \'{e}tale morphism where $H = \hat{\bb{A}}^g = 
\Spf R[\bx]\h$ where $\bx =\{x_1,\dots ,x_g\}$ is a system of coordinate 
functions. By Proposition $1.4$ in \cite{bui95} 
we have the following isomorphism of $\pi$-formal schemes
$$
J^nU \simeq U \times_H J^n H,
$$
for all $n$. 
By the above discussion we have $J^nH \simeq \Spf R[\bx,\bx_1,\dots,\bx_n]\h$
where $\bx_i = \{x_{1i},\dots ,x_{gi}\}$ is a system of coordinates for all $i$.
Therefore we have 
$$
J^nU \simeq U \times_H J^n H \simeq U \times_H \Spf R[\bx,\dots ,\bx_n]\h.
$$
Hence $\{\bx_i\}_{i=0}^n$ form an \'{e}tale coordinate system on $J^nG$ around
the identity section of $J^nG$. Now $N^nG$ is given by the following fiber
product
$$\xymatrix{
N^nG=J^nU\times_US\ar[r]^-{}\ar[d] & J^nU \ar[d]\\
 S \ar[r]^-{e} & U ,
}$$
where the product is taken over the identity section $S\xrightarrow{e} U$.
By composition, we have a section $S\xrightarrow{e}U\rightarrow H$ of $H$, 
which is given by the map of $R$-algebras $R[\bx]\h\rightarrow R$ as 
$\bx \mapsto 0$.
Hence we have $$N^nG =J^nU\times_U S=(J^nH\times_S U)\times_U S=J^nH\times_H S\cong \Spf(R[\bx_1,...,\bx_n]\h).$$
The functions $\{\bx_1,\dots ,\bx_n\}$ form a coordinate system for the 
affine $\pi$-formal group scheme $N^nG$.

For all $n \geq 0$, we have the following extension of smooth commutative
group schemes 
\begin{align}
\label{canexact}
0 \map N^nG \stk{\iota}{\map} J^nG \stk{u}{\map} G \map 0
\end{align}
where $N^nG = \ker( u: J^nG \map G)$ for all $n$. 
Note that the morphism $u$ induces a morphism of $\pi$-formal 
$S$-schemes (still denoted
by $u$) $u: N^n G \map N^{n-1}G$ for all $n \geq 2$. 

Consider the system of $\pi$-formal group schemes $N^*G =\{N^nG\}_{n=1}^\infty$.
By \cite{PS-2}, $N^*G$ is a prolongation sequence with for all $n \geq 2$,
$\fra: N^nG \map N^{n-1}G$ as the associated lift of Frobenius morphism with 
respect to $u$ which is also called the {\it lateral Frobenius}. By 
Theorem $4.3$ of \cite{BS_b} or Theorem $5.3$ of \cite{PS-2}, the morphism
$\fra$ satisfies 
\begin{align}
\label{fphi}
\phi^{\circ 2} \circ \iota = \phi \circ \iota \circ \fra.
\end{align}

We recall Theorem $1.3$ (in the special case of $m =0$) from \cite{PS-2}:

\begin{theorem}
\label{NandJ}
Let $G$ be a $\pi$-formal group scheme over $S$. Then
$$
N^*G \simeq J^*(N^1G),
$$
as a canonical isomorphism of prolongation sequences of $\pi$-formal schemes 
over $S^*$. In particular, for all $n \geq 1$ we have
$$
N^nG \simeq J^{n-1}G.
$$
\end{theorem}
Hence combining the above theorem and (\ref{fphi}) we obtain the following:

\begin{theorem}
\label{NandJphi}
For all $n \geq 2$ we have the following commutative diagram of $\pi$-formal 
group schemes
$$
\xymatrix{
J^{n-1}(N^1G) \simeq N^nG \ar[r]^-{\phi \circ \iota} \ar[d]_\fra & J^{n-1}G 
\ar[d]^\phi \\
J^{n-2}(N^1G) \simeq N^{n-1}G \ar[r]^-{\phi \circ \iota} & J^{n-2}G.
}
$$
\end{theorem} 

We will now discuss some examples of the jet space functor in the case of 
additive and the multiplicative $\pi$-formal schemes.

\begin{enumerate}
\item  {\bf The additive group scheme $\hG$.} Consider $\hG = 
\Spf R[x]\h$, the additive $\pi$-formal 
group scheme over $S$. By the functor of points definition of arithmetic jet 
spaces, we have 
\begin{align}
\label{JnGa}
J^n(\hG) \simeq \bb{W}_n,
\end{align}
where $\bb{W}_n$ is the $\pi$-formal
affine $(n+1)$-plane $\hat{\bb{A}}^{n+1}$ endowed with the additive group law of the 
$\pi$-typical Witt vectors.

The projection map $u: J^n(\hG) \map J^{n-1}(\hG)$ is given by the restriction 
map $T$ and $\phi: J^n(\hG) \map J^{n-1}(\hG)$ is given by
the Frobenius $F$ map of $\pi$-typical Witt vectors.

\item {\bf The multiplicative group scheme $\hGm$.} Consider $\hGm = 
\Spf R[x,x^{-1}]\h$, the multiplicative
$\pi$-formal group scheme over $S$. Again, by the functor of points definition
of arithmetic jet spaces, we have 
$$J^n(\hGm) \simeq \bb{W}_n^\times,$$ 
where $\bb{W}_n^\times$ is the
subfunctor consisting of the multiplicatively invertible elements of $\bb{W}_n$.
We now give an explicit description of $J^n(\hGm)$.

The morphism $\vp: \hGm \map \hGm$ given by $x \map x^q$ is a lift of Frobenius
map compatible with the group structure of $\hGm$. Hence by the universal 
property of the canonical prolongation sequence (\ref{univ}), $\vp$ induces
a canonical splitting of $\pi$-formal group schemes 
$s: \hGm \map J^n(\hGm)$ of the following 
\begin{align}
\xymatrix{
0 \ar[r] & N^n(\hGm) \ar[r]^\iota & J^n(\hGm) \ar[r]^-u & \hGm 
\ar@/_1pc/[l]_-s \ar[r] & 0.
}
\end{align}
Hence we have the following isomorphism
\begin{align}
u \times (\mathbbm{1} - s \circ u) : J^n(\hGm) \simeq \hGm \times N^n(\hGm)
\end{align}
as $\pi$-formal group schemes over $S$. The morphism 
\begin{align}
\label{N1Ga}
\frac{1}{\pi}\log_{{\bb{G}}^{\mathrm{for}}_{\mathrm{m}}}
 \circ (\phi \circ \iota): N^1(\hGm) \longrightarrow \hG
\end{align}
induces an isomorphism of $\pi$-formal schemes, where 
$\log_{{\bb{G}}^{\mathrm{for}}_{\mathrm{m}}}$ is the 
formal logarithm of the formal multiplicative  group law 
${\bb{G}}^{\mathrm{for}}_{\mathrm{m}}$ 
to the additive group law
${\bb{G}}^{\mathrm{for}}_{\mathrm{a}}$ given by 
$$
\log_{ {\bb{G}}^{\mathrm{for}}_{\mathrm{m}}}(T) = \sum_{n=1}^\infty 
(-1)^{n+1}\frac{T^{n}}{n}.
$$
Hence by Theorem \ref{NandJphi} and (\ref{JnGa}), we have,
$N^n(\hGm) \simeq \WW_{n-1}$ for all $n \geq 1$,
 as $\pi$-formal group schemes over $S$. Therefore for all $n\geq 0$, we 
have 
\begin{align}
\label{JnGm}
J^n(\hGm) \simeq \hGm \times \WW_{n-1}
\end{align}
as $\pi$-formal group schemes over $S$. Note that evaluating the above 
at $\FF_p$, recovers the well known isomorphism of groups 
$\ZZ_p^\times \simeq \FF_p^\times \times \ZZ_p$.
\end{enumerate}

\subsection{Delta characters of group schemes}
We recall some basic results on delta characters from Sections $7,8$ of 
\cite{BS_b} that led to the construction of the filtered object
 $\Hd(G)$ over $K$.

Let $T^*$ be a prolongation sequence. For any $s \geq 0$, define the {\it 
shifted prolongation sequence} to be $T^{*+s}= \{T^{s+n}\}_{n=0}^\infty$.
Then a morphism $\Theta: J^nG \map \hG$ is called a {\it delta function of 
$G$ of order $\leq n$}. By the universal property of jet spaces, such a 
$\Theta$ induces a morphism of prolongation sequences $\Theta: J^{*+n}G \map
\hG$.

Define a {\it delta character of order $ \leq n$}, $\Theta:G \map \hG$ 
to be a delta function of order $\leq n$ from $G$ to $\hG$,
which is also a group homomorphism of $\pi$-formal group schemes.
By the universal property of arithmetic jet schemes as in equation (\ref{univ}),
an order $\leq n$ delta character is equivalent to a homomorphism
$\Theta:J^nG \map \hG$ of $\pi$-formal group schemes over $S$. 
We denote the group of 
delta characters of order $n$ by $\bX_n(G)$:
$$
\bX_n(G)=\Hom_S(J^nG,\hG).
$$
A delta character $\Theta$ is said to have {\it exact order} $n$ if 
$\Theta \in \bX_n(G) \backslash u^*\bX_{n-1}(G)$ where $u^*$ is the
pullback of delta characters by $u: J^nG \map J^{n-1}G$.
We will denote the exact order of $\Theta$ by $\mathrm{ord}(\Theta)$.
For any $R$-module $M$, let us denote
$$
M_{\phi} = R\otimes_{\phi,R}M.
$$

Since $\hG$ is an $R$-module $\pi$-formal scheme over $S$, $\bX_n(G)$ naturally becomes an $R$-module. Furthermore, the inverse system $J^{n+1}G\xrightarrow{u} J^{n}G$ defines a directed system of $R$-modules via the pullback map:
$$\ldots \xrightarrow{u^{*}} {\bX}_{n}(G)\xrightarrow{u^{*}} {\bX}_{n+1}(G) \xrightarrow{u^{*}}\ldots$$ 
We define 
$${\bX}_{\infty}(G)=\varinjlim {\bX}_{n}(G).$$

Similarly, pulling back by the Frobenius map $\phi:J^{n+1}G\to J^nG$ induces a Frobenius operator
$\bX_n(G)\to \bX_{n+1}(G)$. However since $\phi:J^{n+1}G\to J^nG$ is not a morphism over $\Spf R$ but
instead lies over the Frobenius endomorphism $\phi$ of $\Spf R$, some care is required.
Consider the relative Frobenius morphism $\phi_{G/R}$, defined to be the unique
morphism making the following diagram commute:
$$\xymatrix{
J^{n+1}G \ar@{.>}^{\phi_{G/R}}[rd] \ar@/^/[rrd]^\phi \ar@/_/[ddr]& & \\
& J^nG \times_{(\Spf R),\phi} \Spf R \ar[d] \ar[r] & J^nG \ar[d] \\
& \Spf R \ar[r]_\phi & \Spf R 
}$$
Then $\phi_{G/R}$ is a morphism of $A$-module formal schemes over $\Spf R$.

Then given a delta character $\Theta:J^nG\to \hG$, define $\phi^*\Theta$ to be 
the composition
\begin{equation}
	J^{n+1}G \longlabelmap{\phi_{G/R}} J^nG \times_{(\Spf R),\phi} \Spf R 
	\longlabelmap{\Theta\times\mathbbm{1}} \hG\times_{(\Spf R),\phi} \Spf R\longlabelmap{\iota} \hG,
\end{equation}
where $\iota$ is the isomorphism of $A$-module formal schemes over $R$
coming from the fact that $\hG$ descends to $\hat{A}$ as an $A$-module scheme. For any $R$-algebra $B$,
the induced morphism on $B$-points is
$$
G(W_{n+1}(B)) \longlabelmap{G(F)} G(W_n(B)^\phi) \longlabelmap{\Theta_B^\phi} 
B^\phi \longlabelmap{b\mapsto b} B.
$$
Note that the composition $G(W_{n+1}(B)) \map B$ is indeed a morphism of 
$A$-modules because the identity map $B^\phi \map B$ is $A$-linear since 
$\phi$ restricted to $\hA$ is the identity.

Hence we have an additive map $\phi^*: \bX_n(G) \to \bX_{n+1}(G)$ given by 
$\Theta\mapsto \phi^*\Theta$. Note that this map is not $R$-linear. 
 However, the map
$$\bX_n(G) \longmap \bX_{n+1}(G)_\phi, \quad \Theta\mapsto 
\phi^*\Theta $$ is $R$-linear. 
Taking direct limits in $n$, we obtain an $R$-linear map
$$ \bX_\infty(G) \longmap \bX_\infty(G)_\phi, \quad \Theta\mapsto
\phi^*\Theta.  $$
Hence $\bX_\infty(G)$ is a left module over the twisted polynomial
ring $R\{\phi^*\}$ with commutation law $\phi^* r = \phi(r)\phi^*$.

Applying $\Hom(-,\hG)$ to (\ref{canexact}) gives us the following exact 
sequence of $R$-modules
\begin{align}
0 \map \bX_0(G) \longrightarrow \bX_n(G) \stk{\iota^*}{\longrightarrow} 
\Hom(N^nG,\hG) \stk{\partial}{\longrightarrow} \Ext(G,\hG)
\end{align}
where $\Ext(G,\hG)$ is the group parametrizing isomorphism classes of 
$\pi$-formal group schemes which are extensions of $G$ by $\hG$ and 
$\partial$ is the connecting map.

Let $\bI_n(G):= \mathrm{image} (\partial)$. The maps $u: J^nG \map J^{n-1}G$ 
naturally induce maps on the kernels $u: N^{n+1}G \map N^{n}G$. Hence via
pull-back $u^*: \Hom(N^nG,\hG) \inj \Hom(N^{n+1}G,\hG)$ is an injection of 
$R$-modules. Therefore we have $\bI_n(G) \subset \bI_{n+1}(G)$. We now define
\begin{align}
\label{bI}
\bI(G) := \varinjlim \bI_n(G).
\end{align}

We define the $R$-module
$$
\bH_{n}(G) = \frac{\Hom(N^{n}G,\hG)}{i^*\phi^*(\bX_{n-1}(G)_{\phi})}.
$$
Note that 
$u:N^{n+1}G\map N^nG$ induces $u^*:\Hom(N^nG,\hG) \map \Hom(N^{n+1}G,\hG)$. 
Moreover, since $u$ commutes with both $i$ and $\phi$, we have
$$
u^*i^*\phi^*(\bX_n(G)) = i^*\phi^*u^*(\bX_n(G)) \subset i^*\phi^*(\bX_{n+1}(G)),
$$
and hence $u$ also induces a map $u^*:\bH_n(G) \map \bH_{n+1}(G)$.
Define 
\begin{equation}
	\label{bigH}
	\Hd(G)= \varinjlim \bH_n(G)
\end{equation}
where the limit is taken in the category
of $R$-modules. Similarly, $\mfrak{f}: N^{n+1}G \map N^nG$ induces 
$\mfrak{f}^*:\Hom(N^nG,\hG) \map \Hom(N^{n+1}G,\hG)$, 
which descends to a $\phi$-semilinear morphism of $R$-modules
\begin{equation}	
	\label{latfrobH}
	\mfrak{f}^*:\bH_n(G) \map\bH_{n+1}(G)
\end{equation}
because (by Proposition $6.3(2)$ in \cite{BS_b}) we have $\mfrak{f}^*i^*\phi^*(\bX_{n-1}(G))= 
i^*\phi^*\phi^*(\bX_{n-1}(G)) \subset i^*\phi^*\bX_{n}(G)$.
This in turn induces a $\phi$-semilinear endomorphism $\mfrak{f}^*:\Hd(G) \map \Hd(G)$. 

We define 
$$
\bXp(G):= \varinjlim \bX_n(G)/\phi^*\bX_{n-1}(A)_\phi.
$$
As in Section $8$ of \cite{BS_b},
we have the following short exact sequence of $R$-modules 
$$
0 \map \bXp(G) \map \Hd(G) \map \bI(G) \map 0,
$$
where $\bI(G) \subset \Ext(G,\hG)$ is as described in equation $(7.3)$ of \cite{BS_b}
(even though the construction in \cite{BS_b} was made for abelian schemes, the above holds true
for a general smooth commutative $\pi$-formal scheme $G$ of finite type over $S$ and this can
be shown without making any changes to the proofs of \cite{BS_b}).
By Theorem $8.7$ of \cite{PS-2}, $\bXp(G)$ is a free $R$-module of rank $g$.
As a consequence we obtain 
\begin{align}
\label{Xprimiso}
\bXp(G) \simeq R\langle \Theta_1,\dots , \Theta_g\rangle,
\end{align}
where $\mathbb{B} = \{\Theta_1,\dots , \Theta_g\}$ is a primitive basis (see page $418$ 
of \cite{BS_b} for definition) for 
$\bX_\infty(G)$ with  $o_1,\dots o_g$ being
 the respective exact orders of the delta 
characters $\Theta_i$ for all $i=1,\dots ,g$. Also by Theorem $1.6$ of \cite{PS-2},
there is a non-zero integer $\mup$ called the {\it upper splitting number} such that  
$\mup = \max\{o_1,\dots o_g\}$ and satisfies $\mup \leq g+1$.

\subsection{Delta characters of additive and multiplicative group schemes}

\begin{enumerate}
\item {\bf Delta characters of $\hG$.}
For all $n \geq 0$, the $R$-module of delta characters of order $\leq n$ is given
by
\begin{align}
\bX_n(\hG) \simeq \Hom (\WW_n,\hG) \simeq R\langle (\mathbbm{1}_{\hG} \circ T^{\circ n}), 
(\mathbbm{1}_{\hG} \circ T^{\circ (n-1)} \circ F), \ldots, 
(\mathbbm{1}_{\hG} \circ F^{\circ n}) \rangle.
\end{align}
where the map $\phi: J^n(\hG) \map J^{n-1}(\hG)$ is identified with the Frobenius
map $F: \WW_n \map \WW_{n-1}$ under the isomorphism $J^n(\hG) \simeq \WW_n$ as in 
(\ref{JnGa}).
Hence we have the following isomorphism of $R\{\phi^*\}$-modules
\begin{align}
\bX_\infty(\hG) \simeq R\{\phi^*\}\langle \mathbbm{1}_{\hG} \rangle.
\end{align}
Hence by (\ref{Xprimiso}), the upper splitting number of $\hG$ is $\mup = 0$.

\item {\bf Delta characters of $\hGm$.}
For all $n \geq 0$, the $R$-module of delta characters of order $\leq n$ is given
by
\begin{align}
\bX_n(\hGm) \simeq \Hom (\hGm \times \WW_{n-1}, \hG) \simeq \Hom(\WW_{n-1},\hG),
\end{align}
where the first isomorphism is induced from (\ref{JnGm}) and the second one is 
because of the fact that $\Hom(\hGm,\hG) =\{0\}$. Hence we have 
\begin{align}
\bX_\infty(\hGm) \simeq R\{\phi^*\}\langle \mathbbm{1}_{N^1(\hGm)} \rangle,
\end{align}
and we have $N^1(\hGm) \simeq \hG$ from (\ref{N1Ga}). Hence 
the delta character $\mathbbm{1}_{N^1(\hGm)}$ is of exact order $1$. Therefore by
(\ref{Xprimiso}), the upper splitting number $\mup =1$.
\end{enumerate}

\section{Frobenius relation between Jet spaces and their Kernels}\label{SFrob}
We will now recall some of the theory of arithmetic jet spaces for smooth 
commutative group schemes, most of which can be found in \cite{BS_b} and
\cite{PS-2}.
Let $G$ be a smooth finite type commutative $\pi$-formal group scheme of 
relative dimension $g$ over $S$ and suppose $r= \rk_R \Ext(G,\hG)$.

We recall the following results from \cite{BS_b}. We note that even though
the results in \cite{BS_b} are stated for abelian schemes, they are in fact
true for any smooth commutative $\pi$-formal group scheme of finite type
over $S$ and the arguments for the 
latter remain unchanged from the case for abelian schemes.

\begin{theorem}
\label{comfact}
Let $G$ be a smooth group scheme over $S$. Then
the morphism of group schemes 
$(\phi \circ \iota - \iota \circ \fra) : N^nG \map J^{n-1}G$ uniquely factors through  $N^1G$ as
$$\xymatrix{
N^nG \ar[rr]^-{\phi \circ \iota - \iota \circ \fra}  
\ar[d]_u & & J^{n-1}G \\
N^1G \ar[rru]_\Delta & &
}$$
\end{theorem}
\begin{proof}
See Theorem $4.4$ in \cite{BS_b}.
\end{proof}

For any delta character $\Theta \in \bX_n(G)$, consider the derivative 
map $D\Theta : T_0(J^nG) \map T_0(\hG)$ between the tangent spaces at the 
identity sections. With respect to the local Witt coordinates chosen in Section
\ref{localcoordinates}, let the derivative map be given by
\begin{align}
\label{differential}
D\Theta = (A_0 \cdots A_n)
\end{align}
where $A_j \in \mb{Mat}_{1\times g} (R)$.

Let $\BB= \{\Theta_{1},\dots, \Theta_{g}\}$ be a primitive basis for 
$\bX_\infty(G)$  such that 
$$
\bXp(G) \simeq R\langle \Theta_1,\dots , \Theta_g\rangle,
$$
as in equation (\ref{Xprimiso}) with
$o_1,\dots, o_g$ be the respective exact orders of the delta characters 
$\Theta_i$ for all $i=1,\dots , g$. Hence for all $i$, the delta character
$\Theta_i$ is a group homomorphism
$$
J^{o_i}G \map \hG.
$$
For each $\Theta_i$, let the derivative
matrix at the identity, as in (\ref{differential}), be 
$D\Theta_i=(A_{0i} \cdots A_{o_ii})$ where $A_{ji}$ are $(1\times g)$-matrices.
Consider the $g \times g$-matrix
$$
\tilde{\Gamma} := \left(\begin{matrix}A_{01} \\ \vdots\\  A_{0g}\end{matrix}\right) \in 
\Mat_{g\times g}(R).
$$

\begin{proposition}
\label{diff}
Let $\Theta$ be a delta character in $\bX_n(G)$.
\begin{enumerate}
        \item We have
                $$
                \iota^*\phi^*\Theta = \mfrak{f}^*(\iota^*\Theta)+ 
\gamma_\Theta. \Psi_1,
                $$
                where $\gamma_\Theta=\pi A_0$. 
        \item For $n\geq 2$, we have
                $$
                \iota^*(\phi^{\circ n})^*\Theta= (\mfrak{f}^{n-1})^* \iota^*\phi^*\Theta.
                $$
\end{enumerate}
\end{proposition}
\begin{proof}
See Proposition $6.3$ in \cite{BS_b}
\end{proof}

Suppose $\bb{B} = \{\Theta_1,\dots ,\Theta_g\}$ be a primitve $R$-basis 
of $\bXp(G)$.
For all $i= 1, \dots g$, we have $\iota^*\phi^*\Theta_i = \fra^*\iota^* \Theta_i +
\gamma_i \Psi_1$ for some $\gamma_i :=\gamma_{\Theta_i} \in \Mat_{1\times g}(R)$. 
By Proposition 
\ref{diff} (1) we have the following equality of $g \times g$-matrices
\begin{align}
\label{Gammamatrix}
{\Gamma}:=\left(\begin{matrix} \gamma_1 \\ \vdots \\ \gamma_g  \end{matrix}\right) = \pi 
\tilde{\Gamma}.
\end{align}

By Theorem \ref{comfact}, pulling back via $\Delta$ we obtain the following
map of $R$-modules
\begin{align}
\xymatrix{
\bX_n(G) \ar[r]^-{\Delta^*} & \Hom(N^1G,\hG),
}
\end{align}
given by (as in Proposition \ref{diff} (1)),
\begin{align}
\label{Thegamma}
\Theta \mapsto \gamma_\Theta \Psi_1.
\end{align}

Consider the morphism of $\pi$-formal group schemes
\begin{align}
\xymatrix{
N^1G \ar[r]^{\phi \circ \iota} & G.
}
\end{align}
The above induces the corresponding injective linear map on the Lie algebras
\begin{align}
\label{Liepi}
\xymatrix{
\Lie N^1G \ar[r]^-{D(\phi \circ \iota)} & \Lie G, \mbox{ given by } \bt \mapsto
\pi \bt
}
\end{align}
which upon tensoring over $K$ becomes an isomorphism.

We define the map $\tilde{\Upsilon}_n: \bX_n(G) \map (\Lie G)^*$ via the 
following composition:
\begin{align}
\xymatrix{
\Hom (\Lie N^1G,\hG)_K & \Hom (\Lie G, \hG)_K \simeq (\Lie G)^*_K
\ar[l]_-{D(\phi \circ \iota)}^-\sim\\
\Hom(N^1G, \hG) \ainj{u} & \bX_n(G). \ar[l]_-{\Delta^*} 
\ar[u]_{\tilde{\Upsilon}_n}
}
\end{align}
By Proposition \ref{diff} (2), $\tilde{\Upsilon}_n$ restricted to 
$\phi^*(\bX_{n-1}(G))$ is $0$ and hence induces the map of $R$-modules
$$ 
\tilde{\Upsilon}_n: \bX_n(G)/ \phi^*\bX_{n-1}(G) \map (\Lie G)^*_K.
$$
Hence passing to the limit over $n$ and combining (\ref{Thegamma}) and 
(\ref{Liepi}) we obtain 
\begin{align}
\label{Upsilon-new}
\Upsilon : \bXp(G) &\longrightarrow (\Lie G)^*_K \\
[\Theta] & \longmapsto \frac{\gamma_\Theta}{\pi} \Psi_1, 
\nonumber
\end{align}
for any class $[\Theta] \in \bXp(G)$. By Proposition \ref{diff} (1), 
$\frac{\gamma_\Theta}{\pi} \in \Mat_{1 \times g}(R)$ and hence therefore
the image of $\Upsilon$ inside $(\Lie G)^*_K$ is integral, in other words, 
we have $\Upsilon: \bXp(G) \longrightarrow (\Lie G)^*$.

Hence with respect to the primitive
basis $\bb{B} = \{\Theta_1, \dots , \Theta_g\}$ of $\bXp(G)_K$, the map 
$\Upsilon_K$ is given by the matrix 
\begin{align}
\label{Upsilonmatrix}
[\Upsilon_K] = \frac{1}{\pi} \Gamma = \tilde{\Gamma}. 
\end{align}

\subsection{Canonical splitting of Lie algebras}
\label{Can-split-Lie}
Recall for any $R$-algebra $B$, the ghost map 
$$
w: W_n(B) \map \prod_\phi^n B,
$$
introduced in Section \ref{pityp}. Hence applying the functor $G(-)$ to the 
above and by the functorial definition of jet spaces, we obtain the 
following morphism of $\pi$-formal group schemes
\begin{align}
\label{ghmap}
\xymatrix{
J^nG \ar[d]_u \ar[r]^w & \prod_\phi^n G\ar[d]^{\mathrm{pr}_0} \\
G \ar@{=}[r] & G.
}
\end{align}
Suppose $\bx_0$ be an \'{e}tale coordinate system around the identity section 
of $G$ and $(\bx_0,\dots , \bx_n)$ the induced Witt coordinate system around 
the identity section of $J^nG$. Then the ghost map $w$ in terms of the above
coordinates is given by
\begin{align}
\label{wmap}
(\bx_0,\dots ,\bx_n) \mapsto \langle \bx_0, \bx_0^q+\pi \bx_1, \dots ,
\bx_0^{q^n} + \pi \bx_1^{q^{n-1}} + \cdots + \pi^n \bx_n\rangle.
\end{align}

Taking derivative at the identity section of diagram (\ref{ghmap}) gives us
the following commutative diagram of $R$-modules
$$
\xymatrix{
\Lie J^nG \ar[d]_u \ar[r]^-{Dw} & \Lie G \times \cdots (\Lie G)_{\phi^n} 
\ar[d]_{\pr_0} \\
\Lie G \ar[r]^= & \Lie G, \ar@/_1pc/[u]_{s_0}
}
$$
where $s_0(\bx_0) = (\bx_0, 0,\dots , 0)$ is the natural section of $\pr_0$. 
Hence from (\ref{wmap}), the derivative map $Dw$ is given by
\begin{align}
(\bx_0,\dots ,\bx_n) \mapsto \langle \bx_0, \pi \bx_1, \dots , \pi^n \bx_n
\rangle
\end{align}

Note that $Dw$ is an injective map of $R$-modules and also the image of the
section
$s_0$ is contained inside the image of $Dw$. Consider the map of Lie algebras
$v:=(Dw)^{-1} \circ s_0: \Lie G \map \Lie J^nG$ which is a section of the 
following short exact sequence of $R$-modules obtained from (\ref{canexact})
\begin{align}
\label{Liesplit}
\xymatrix{
0 \ar[r] & \Lie N^nG \ar[r]^{D\iota} & \Lie J^nG \ar[r]^{Du} 
\ar@/_1.5pc/[l]_\switt &  
\Lie G \ar@/_1.5pc/[l]_v \ar[r] & 0,
}
\end{align}
where $\switt = \mathbbm{1} - v \circ Du$.
In terms of the \'{e}tale coordinate system, $v$ is given by 
\begin{align}
\bx_0 \mapsto (\bx_0,0 ,\dots, 0).
\end{align}

\subsection{Extension of group schemes.} 
\label{Ext-of-gp-sch}
We now recall some of the theory of extensions of arithmetic jet spaces 
introduced in Section $8$ in \cite{BS_b}. 
Given a character $\Psi \in \Hom(N^nG, \hG)$, consider the push-out of 
the short exact sequence of $\pi$-formal group schemes in diagram 
(\ref{canexact}) via $\Psi$:
$$
\xymatrix{
0 \ar[r] & N^nG \ar[d]_\Psi \ar[r]^\iota & J^nG \ar[d]_{g_\Psi} 
\ar[r]^u & G \ar@{=}[d] \ar[r] & 0\\
0 \ar[r] & \hG \ar[r]^-\iota & \Psi_* (J^nG) \ar[r] & G \ar[r] & 0,
}
$$
where $\Psi_*(J^nG) = \frac{J^nG \times \hG}{\Gamma(N^nG)}$ and 
$\Gamma(N^nG) = \{(\iota(z), -\Psi(z)) | z \in N^nG\} \subset J^nG \times N^nG$
and the induced morphism $g_\Psi: J^nG \map \Psi_*(J^nG)$ is given by
$g_\Psi(x) = [x,0] \in \Psi_*(J^nG)$.

The splitting of the Lie algebras in (\ref{Liesplit}), naturally induces 
the splitting of Lie algebras of the push-out extension
$$
\xymatrix{
0 \ar[r] & \Lie \hG \ar[r]^-{D\iota} & \Lie \Psi_*(J^nG) 
\ar@/_1.25pc/[l]_{s_\Psi} \ar[r] & \Lie G \ar[r] & 0,
}
$$
where 
$$
s_\Psi: \Lie \Psi_*(J^nG) = \frac{\Lie J^nG \times \Lie \hG}{\Gamma(N^nG)}
\longrightarrow \Lie \hG
$$
is given by 
$$
s_\Psi([\bt,y]) = D\Psi(\switt(\bt)) + y.
$$

Suppose $A$ is an abelian scheme over $S$. Recall that $\Ext^\sharp(A, \hG)$ 
parametrizes isomorphism classes of extensions of $A$ by $\hG$,
along with a splitting of the corresponding short exact sequence of the 
(commutative) Lie algebras. (See \cite{MazMess}, p. 13–14, where it 
would be denoted $\mathrm{Extrig}(A, \hG)$.) The $R$-module $\Ext^\sharp(A,\hG)$
 is canonically isomorphic to the de Rham cohomology $\Hdr(A)$
of the abelian scheme $A$. Consider the Hodge sequence associated to $A$
$$
0 \map H^0(A,\Omega_A) \map \Ext^\sharp(A,\hG) \map \Ext(A,\hG) \map 0.
$$

We have a natural map of $R$-modules 
$$
\tilde{\Phi}: \Hom(N^nA,\hG) \longrightarrow \Ext^\sharp(A,\hG)
$$
given by
$$
\Psi \mapsto [(\Psi_*(J^nA),s_\Psi)].
$$
Hence $\tilde{\Phi}$ induces $\tilde{\Upsilon}: \bX_n(A) \map H^0(A,\Omega_A)$ 
as follows:
\begin{align}
\xymatrix{
0 \ar[r] & \bX_n(A) \ar[d]_{\tilde{\Upsilon}} \ar[r] & \Hom (N^nA,\hG) 
\ar[d]_{\tilde{\Phi}}
 \ar[r]^-\partial & \Ext(A,\hG) \ar@{=}[d]\\
0 \ar[r] & H^0(A,\Omega_A) \ar[r] & \Ext^\sharp(A,\hG) \ar[r] & \Ext(A,\hG) 
\ar[r] & 0
}
\end{align}

By Proposition $6.1~(3)$ in \cite{BS_b}, $\tilde{\Phi}$ restricted to 
$\iota^*\phi^*(\bX_{n-1}(A))$ is $0$. Taking quotients and passing to the 
limit by varying $n$, we obtain the following map of short exact sequences 
(c.f diagram ($8.1$) in \cite{BS_b})
\begin{align}
\label{short}
\xymatrix{
0 \ar[r] & \bXp(A) \ar[d]_\Upsilon \ar[r] & 
\bH_\delta(A) \ar[d]_\Phi \ar[r] &\bI(A) \ar@{^{(}->}[d] \ar[r] &  0 \\
0 \ar[r] & H^0(A,\Omega_A) \ar[r] & \Hdr(A) \ar[r] & H^1(A,\Ou_A)
 \ar[r] & 0, 
}
\end{align} 
where $\bI(A)$ is as in (\ref{bI}).
Note that the map $\Upsilon$ above coincides with 
\eqref{Upsilon-new} (see the proof of Theorem $8.2$ of \cite{BS_b}) in
the case of abelian schemes.

%\cred
%However the following semi-abelian generalised version will hold without any change in the argument. 

%\begin{theorem}
%\label{upsilon}
%The diagram (\ref{short}) is a map of short exact sequences of $K$-modules. 

%The map $\Upsilon$ is an isomorphism of $K$-vector spaces if and only if
%the $g\times g$ matrix $\Gamma$ is invertible over $K$. Hence $\Phi$ is injective if and only if $\Gamma$ is invertible over $K$.
%\end{theorem}
In Section \ref{SUpsilon}, we will show that $\Gamma$ is invertible over $K$, eventually to prove Theorem \ref{isoupsilon}. Towards that we need some preparation.
\cblue

\section{$N^1X$ and reduction mod $\pi$ of $X$}\label{mod-p}

A $\pi$-formal scheme $X$ with a fixed morphism $e:S \map X$ will be called
a {\em pointed $\pi$-formal scheme}. Let the tuple $(X,e)$ be a pointed $\pi$-formal scheme. For all $n \geq 0$, consider $N^{n}X = J^{n}X \times_{X} S$ which is 
the following fiber product
$$\xymatrix{
J^{n}X \ar[d] & N^{n}X \ar[l]_{i} \ar[d] \\
X & S. \ar[l]_-{e} 
}$$
\cblue
In this section we will establish a functorial isomorphism between
$N^1X$ and the set of points of $X$ that reduce to a fixed point in $\ov{X}$
over the residue field $l$ of $R$. To do so we need some preparation.

Let $B$ be an $R$-algebra and $\ov{B}:= B\otimes_R \frac{R}{\pi R}$. 
Then we have
\begin{align}
\label{Frob}
\xymatrix{
W_n(B)\ar[d]_F \ar[r] & W_n(\ov{B}) \ar[d]^{F}\\
W_{n-1}(B) \ar[r] & W_{n-1}(\ov{B})
}
\end{align}  

Also for all $n$ we have
\begin{align}
\label{Fbar}
\xymatrix{
W_n(\ov{B}) \ar[r]^-T \ar[d]_F& \ov{B} \ar[d]^{\ov{F}} \\
W_{n-1}(\ov{B}) \ar[r]^-T & \ov{B}
}
\end{align}
where $\ov{F}(y) = y^q$ is the absolute Frobenius.

Combining (\ref{Frob}) and (\ref{Fbar}) we have 
$$
\xymatrix{
W_n(B) \ar[r] \ar[d]_F  & \ov{B} \ar[d]^{\ov{F}} \\
W_{n-1}(B) \ar[r] & \ov{B}
}
$$
Hence we obtain
\begin{align}\label{Frob-X}
\xymatrix{
J^nX(B)\ar[r]^-\sim \ar[d]^\phi & X(W_n(B)) \ar[r] \ar[d]^{X(F)} & X(\ov{B})
\ar[d]^{X(\ov{F})} \ar[r]^\sim & \ov{X}(\ov{B})\ar[d]^{\bar{\phi}}\\
J^{n-1}X(B)\ar[r]^-\sim & X(W_{n-1}(B)) \ar[r] & X(\ov{B}) \ar[r]^\sim
 & \ov{X}(\ov{B})
}
\end{align}

Given a pointed $\pi$-formal scheme $(X,e)$, let $\bar{e}:\ov{S} \map \ov{X}$ be the 
corresponding morphism after base changing to $\ov{S}$.
Consider the subfunctor $X(\pi) \subset X$ given by
\begin{align}
\label{piX}
 X(\pi)(C):= \{ v \in X(C)|~ \bar{v} = \bar{e} \}
\end{align}
for any $\pi$-adically complete $R$-algebra $C$.

Let $G$ be a group scheme and let $\ov{G}:= G \times_S \ov{S}$ and $P(t) \in
\Z[t]$ be the characteristic polynomial satisfied by $\bar{\phi} \in 
\End(\ov{G})$. Then we have 
\begin{align}
\xymatrix{
J^nG(B) \ar[r] \ar[d]_{P(\phi)} & \ov{G}(\ov{B}) \ar[d]^{P(\bar{\phi})} \\
G(B) \ar[r] & \ov{G}(\ov{B})
}
\end{align}
Then we have $\mathrm{Im}(P(\phi)) \subseteq G(\pi)$.

Consider the subfunctor $(\phi \circ i)(N^1X) \subset X$ given by 
\begin{align}
\label{Imphi}
(\phi \circ i)(N^1X)(C)= \{v \in X(C) |~\exists~ w \in N^1X(C)
\mb{ such that } v= (\phi \circ i) \circ w \}
\end{align}
for all $\pi$-adically complete $R$-algebra $C$.
Then note that $(\phi \circ i)(N^1X) \subset X(\pi)$.

\begin{proposition}
\label{Nniso}
Let $f:U \map V$ be an \'{e}tale morphism of pointed affine $\pi$-formal 
schemes. Then
$N^n U \simeq N^nV$. 
\end{proposition}
\begin{proof}
Let $U = \Spf A$ and $V= \Spf B$. Then $J^nU \simeq J^nV \times_V U \simeq
J_nB \hat{\otimes}_B A$. Then
\beqar
N^nU & \simeq & J^nU \times_U S \\
&=& \Spf~ (J_nB \hat{\otimes}_B A \hat{\otimes}_A R) \\
&=& \Spf~ (J_nB \hat{\otimes}_B R) \\
& =& N^nV.
\eeqar
\end{proof}

Let $V= \Spf R[\bx]\h$ be a pointed scheme with $\bx=(x_1,\dots, x_d)$ 
with the marked $S$-point given by 
the $R$-algebra morphism $e^*:R[\bx]\h \map R,~ \bx \mapsto 0$. Let 
$\bar{e}:k[\bx] \map k$ be the corresponding reduction modulo $\pi$, 
$\ov{S}$-point on $\ov{X}$. Then 
$$
V(\pi) (C) = (\pi C)^d
$$
for all $\pi$-adically complete $R$-algebra $C$.
Also note that 
$$
((\phi\circ i)(N^1V))(C) = \{v:R[\bx]\h \map C |~ v = (\phi\circ i)^* \circ w
\mb{ for some } w: R[\bx_1]\h \map C\}
$$
and $(\phi \circ i)^*: R[\bx]\h \map R[\bx_1]\h$ is given by $\bx \map \pi 
\bx_1$. Therefore for any $\pi$-torsion free $C$ we have 
$((\phi \circ i)(N^1V))(C) = (\pi C)^d$. Hence we have an isomorphism
\begin{align}
\label{phipiiso}
(\phi \circ i):(N^1V) \simeq V(\pi),
\end{align}
since $V$ is smooth over $S$.
\begin{proposition}
\label{equalpi}
Let $f:U \map V$ be an \'{e}tale morphism of pointed affine $\pi$-formal 
schemes. Then $U(\pi) = V(\pi) $.
\end{proposition}
\begin{proof}
Let $U= \Spf A$ and $V=\Spf B$ and $U(\pi) \map V(\pi)$ be the natural map 
between functors. Let $g \in (V(\pi))(C/\pi^n C)$ for any $\pi$-adically 
complete $R$-algebra. Then we have
$$
\xymatrix{
U \ar[d]_-f & \Spf (C/\pi C) \ar[l]_-{\bar{e}} \ar@{^{(}_->}[d] \\
V & \Spf (C/\pi^n C) \ar[l]^-g
}
$$
Since $f:U \map V$ is \'{e}tale, $g$ has a unique lift $h: \Spec(C/\pi^n C)
\map U$ making the above diagram commutative for all $n \geq 1$. Since $C$
is $\pi$-adically complete, this proves our result.
\end{proof}

\begin{proposition}
\label{phiiso}
Let $X$ be a pointed smooth $\pi$-formal $S$-scheme. Then the morphism 
$\phi \circ i$ induces the isomorphism
$$
(\phi\circ i): (N^1X) \simeq X(\pi). 
$$
\end{proposition}
\begin{proof}
Since $X$ is smooth over $S$, there exists an affine open subscheme $U\subset
X$ containing $e$ such that there exists an \'{e}tale morphism $f:U \map V$ 
such that $V = \Spf R[\bx]\h$, where
$\bx=(x_1,\dots, x_d)$ and $d$ is the relative dimension of $X$ over $S$.
Note that $N^1X \simeq N^1U$ and $X(\pi) \simeq U(\pi)$. By 
Proposition \ref{Nniso} we have $N^1X \simeq N^1V$. Then the result follows 
from (\ref{phipiiso}) and Proposition \ref{equalpi}.
\end{proof}

\cb{
Let $\Gfor$ be the formal group law associated to a group scheme $G$ of relative 
dimension $g$ over $S$. Then
by Corollary $11.1.6$ of \cite{Haz}, let 
\begin{align}
\exp_G : TG^\mathrm{for} \simeq (\bb{G}^\mathrm{for}_a \otimes K)^g 
\map \Gfor \otimes K
\end{align}
denote the exponential map associated to the formal group law $\Gfor$,
which is an isomorphism of formal group laws with $\log_G$ as its inverse.

\begin{theorem}
\label{formaliso}
Let $\beta:G \map G$ be an endomorphism of group schemes and let $\beta: \Gfor
\map \Gfor$ still denote the induced endomorphism on the associated formal
group law $\Gfor$. Then we have the following commutative diagram of 
formal group laws:
$$\xymatrix{
\Gfor \otimes K \ar[d]_{\log_G} \ar[r]^\beta & \Gfor \otimes K \ar[d]^{\log_G}\\
(\hGa \otimes K)^g \ar[r]^{D\beta} & (\hGa \otimes K)^g.
}$$
In particular we have 
$$
\log_G(\beta(\bx)) = D\beta \log_G (\bx)
$$
for all $\bx \in \Gfor \otimes K$.
\end{theorem}

\begin{proof}
Let $\bx$ be a coordinate system around the identity section of $\Gfor$. We
denote by the same $\bx$, the induced coordinate on $TG^{\mathrm{for}}$
giving the isomorphism $TG^{\mathrm{for}} \simeq (\hGa)^g$.

In terms of the coordinate $\bx$, we have
$$
\beta(\bx) = B\bx  \pmod {\bx^2}  
$$
where $B = D \beta$ is the derivative matrix of $\beta$ at the identity section 

Also we have 
\begin{align}
\nonumber\exp_G(\bx) &= \bx \pmod{\bx^2} \\
\nonumber \log_G(\bx) &= \bx \pmod{\bx^2} 
\end{align}
where the above maps are inverses of each other in the category of
formal group laws over $K$. Then consider the composition

\begin{align}
\nonumber
(\hGa \otimes K)^g \stk{\exp_G}{\longrightarrow} \Gfor \otimes K 
\stk{\beta}{\longrightarrow}
 \Gfor \otimes K \stk{\log_G}{\longrightarrow} (\Gfor \otimes K)^g.
\end{align}
We have 
$\gamma(\bx):= \log_G \circ \beta \circ \exp_G (\bx) = B\bx \pmod{\bx^2} $. 
Since $\gamma$ is an endomorphism of $(\hGa)^g$, $\gamma$ must be a linear 
operator. 
Therefore $\gamma (\bx) = B \bx $ and hence we have
\begin{align}
\log_G \circ \beta \circ \exp_G &= D\beta. \nonumber 
\end{align}
Since $\log_G$ and $\exp_G$ are inverses to each other, we obtain
\begin{align}
 \log_G \circ \beta  &= D\beta \circ \log_G \nonumber 
\end{align}
and we are done.
\end{proof}

Let $\bx$ be an \'{e}tale coordinate of $G$ around the identity section. 
Then $\phi \circ \iota: N^1 G \map G(\pi)$ is given $\bx \mapsto \pi \bx$. 
Let $\Psi_1: N^1G \map \hG^g$ be the delta character as constructed 
in Proposition $4.1$ in \cite{Buiumcomp} which is given as
\begin{align}
\label{Psi1}
\Psi_1(\bx) := \pi^\nu\log_G \circ (\phi \circ \iota) (\bx) =
\pi^\nu \log_G(\pi \bx)
\end{align}
where $\nu$ is as in the proof of Proposition $4.1$ in \cite{Buiumcomp}.
By Theorem \ref{formaliso} we have 
\begin{align}
\nonumber
\pi^\nu \log_G(\beta(\pi \bx)) &= D\beta\left(\pi^\nu\log_G(\pi\bx) \right) \\
\nonumber \Psi_1 \circ \beta (\bx) &= D\beta \circ \Psi_1(\bx)
\end{align}
Hence we have the following commutative diagram
\begin{align}
\label{phider}
\xymatrix{
N^1G \ar[d]_{\phi \circ \iota} \ar@/_4pc/[dd]_{\Psi_1}\ar[r]^\beta & N^1G 
\ar@/^4pc/[dd]^{\Psi_1} \ar[d]^{\phi \circ \iota} \\
G(\pi) \ar[d]_{\pi^\nu\log_G} \ar[r]^\beta & G(\pi)  
\ar[d]^{\pi^\nu\log_G} \\
\hG^g \ar[r]^{D\beta} & \hG^g.
}
\end{align}
}

\section{The map $\Upsilon$}
\label{SUpsilon}

One of the objectives in this section will be to prove that the map
$\Upsilon_K$ as defined in \eqref{Upsilon-new} induces a $K$-linear
isomorphism between $\bXp(G)_K$ and $(\Lie G)^*_K:= \Hom(\Lie G,\hG)_K$ 
when $G$ is a semi-abelian $\pi$-formal scheme over $S$. 
 
 From now on we consider $R = W(\bF_q)$ where $q$ is a power of our 
prime $p$ with $k = \bF_q$ as its residue field and the fixed lift of 
Frobenius $\phi$ on $R$ is taken to be the identity map. 
Hence a semilinear
operator for an $F$-isocrystal over $K$ are linear operators.

However, while
working with various identities below, the reader will find that we 
still write the endomorphism $\phi$ for coefficients belonging to
 $R$ even though it is the identity on it
(as for example see equation ($\ref{phiup}$) below). This is done to remember
when the coefficients in $R$ are $\phi$-twisted to have better clarity
as quite a few of such expressions remain true for a more general discrete
valuation ring $R$ where the Frobenius $\phi$ does not act as an 
identity operator.

From Theorem \ref{comfact} and Proposition \ref{diff} the morphism $\Delta :
N^1G \map J^{2g}G$ gives us for all $j=1,\dots , g$ the pull-back of the 
primitive characters  
$$\Delta^*\Theta_j:N^1G \map \hG$$
that are given by
\begin{align}
\label{commute}
\Delta^*\Theta_j = \gamma_j.\Psi_1
\end{align}
where $\gamma_j = \pi A_{0j}$. 
For all $n \geq m_{\mathrm{u}}$, and any $f \in \bX_n(G)$ we have
\begin{align}
f = \sum_{j=1}^r \left(\sum_{s=0}^{o_j} c_{sj}\phi^{\circ(n-s)*}\Theta_{i_j}
\right)
\end{align}
Then by Proposition \ref{diff} we get
\begin{align}
\label{phiup}
\Delta^*f &
= (\phi(c_{01}) \gamma_{i_1} + \cdots + \phi(c_{0r})\gamma_{i_r}).\Psi_1.
\end{align}

Let $\oG$ denote the special fiber of $G$ over the residue field $l$ of $R$ and
$F$ be the $q$-th power Frobenius on $\oG$. Suppose 
$$
P(t)= t^{r} + b_{r-1}t^{r-1} + \cdots + b_1 t + b_{0}
$$ be a polynomial 
in $\bZ[t]$ such that $P({F})=0$ in $\End(\oG)$.

For any $\pi$-adically complete $R$-algebra $B$ we have the following 
commutative diagram
\begin{equation}\label{frob-diag}
\xymatrix{
\bb{W}_{r}^g(B) \ar[r]^-\sim &J^{r}(\hG^g)(B) \ar[d]_-{P(\mfrak{f})} 
 & N^{r+1}G(B) \ar[l]_{J^{r}(\Psi_1)}
\ar[r]^-{\phi\circ \iota} \ar[d]_-{P(\mfrak{f})} & J^{r}G(B) \ar[r] 
\ar@{.>}[dl]_\mfg \ar[d]_-{P(\mfrak{\phi})} & \oG(B\otimes k) 
\ar[d]_-{P({F})} \\
& \hG^g(B) \ar[d]_{\mathrm{pr_j}} & N^1G(B) \ar[l]_{\Psi_1}
\ar[r]^-{\phi \circ \iota} & G(B) \ar[r]^-h & \oG(B\otimes k)\\
& \hG (B) & & & 
}
\end{equation}
where $\mathrm{pr}_j: \hG^g \map \hG$ is the projection on the $j$-th component
for all $i=1,\dots,g$.
Note that $\mathrm{Im(\phi \circ \iota)} = \mathrm{Ker}(h)$. Since $P(F)=0$,
we have that $P(\phi)$ factors through $\mfg:J^{r}G(B) \map N^1G$ 
by Proposition \ref{phiiso}.

\cblue
In particular, when $G$ is an abelian scheme of dimension $g$, we have the diagram as above for the characteristic polynomial $P(t)= t^{2g} + b_{2g-1}t^{2g-1} + \cdots + b_1 t + q^g$ of the $q$-th power Frobenius on $\overline{G}$. 

For any group scheme $\ov{G}$ over $\Spec k$, let $P_{\ov{G}}(t)$ denote the 
characteristic polynomial which is satisfied by the absolute $q$-th power 
Frobenius ${F}_{\ov{G}}$ on $\ov{G}$, that is $P_{\ov{G}}({F}) = 0$.

\begin{proposition}
\label{Frob-char}
Let $\ov{G}$ be a semi-abelian scheme over $\Spec k$ satisfying the following 
extension of smooth group schemes
$$
0 \map \ov{T} \stk{f}{\map} \ov{G} \stk{g}{\map} \ov{A} \map 0,
$$
where $\ov{T}$ is a multiplicative torus and $\ov{A}$ is an abelian scheme 
over $\Spec k$. Then
$$
P({F}_{\bar{G}}) = 0 
$$
where $P(F) = P_{\bar{T}}(F) \circ P_{\bar{A}}(F)$.

In particular, we have 
$$
P(F) = F^{2g+1} + (b_{2g-1}-q)  F^{2g} + \cdots + (q^g-qb_1) F - q^{g+1}.
$$
\end{proposition}
\begin{proof}
Consider the following diagram satisfied by the absolute Frobenius $\bar{F}$
$$\xymatrix{
\ov{G} \ar[d]_{P_{\bar{A}}(\bar{F})} \ar[r]^g & \ov{A} 
\ar[d]^{P_{\bar{A}}(\bar{F})}\\
\ov{G} \ar[r]^g & \ov{A}. 
} $$
For all $x \in \bar{G}$ we have
 $f \circ P_{\bar{A}}(F)(x) = P_{\bar{A}}(F)(f({x})) = 0$. Therefore 
$P_{\bar{A}}(x) \in \mathrm{Image} (f)$. Hence we have 
$P_{\bar{T}}(F)(P_{\bar{A}}(F)(x)) = 0$ and this completes the proof of the first 
part.

The second part follows from the fact that $P_{\bar{T}}(F)= F - q$ and 
$P_{\bar{A}}(F) = F^{2g} + b_{2g-1}F^{2g-1} + \cdots + b_1 F + q^g$ for some 
$b_i \in \Z,~ i = 1,\dots , 2g-1$ and $g = \dim \bar{A}$. Then for all $x \in
\ov{G}$ we have
\beqar
P(F)(x) &=& P_{\bar{T}}(F)(P_{\bar{A}}(F)(x)\\
&=& P_{\bar{T}}(F)
(F^{2g}(x) + b_{2g-1}F^{2g-1}(x) + \cdots + b_1 F(x) + q^g(x))\\
&=& (F - q)
(F^{2g}(x) + b_{2g-1}F^{2g-1}(x) + \cdots + b_1 F(x) + q^g(x))\\
&=&
(F^{2g+1} + (b_{2g-1}-q)F^{2g} + \cdots + (q^g-qb_1)F - q^{g+1})(x)
\eeqar
and this proves our result.
\end{proof}

\color{black}

\begin{theorem}
\label{isoupsilon}
Let $G$ be a $\pi$-formal smooth commutative group scheme of finite type 
over $S$ such that the absolute Frobenius $F$ on $\ov{G}$ satisfies $P(F) =0$ where 
$$
P(F) = F^m + b_{m-1} F^{m-1} + \cdots + b_1 F + q^h
$$
for some $h \in \bb{Z}_{\geq 0}$.
%\color{blue}
Then $\Gamma \in \Mat_{g \times g}(K)$ as in (\ref{Gammamatrix}) is invertible.

\color{black}

\end{theorem}

\begin{proof}

For all $j=1,\dots, g$, define the delta character of order $2g$ as 
$\tTheta_j:J^{2g}G \map \hG$ given by
\begin{align}
\tTheta_j:= \mathrm{pr}_j \circ \Psi_1 \circ \mfg.
\end{align}
Then by the above diagram we have 
\begin{align}
(\phi \circ \iota)^* f_j = \pr_j \circ \Psi_1 \circ P(\fra).
\end{align}
Note that the morphism $\Psi_1 \circ P(\fra): N^{2g+1}G \map \hG^g$ is given
by 
\begin{align}
\Psi_{m+1}+ b_{m-1}\Psi_{m} + \cdots + q^h \Psi_1.
\end{align}

Hence we have
\beqar
(\phi \circ \iota)^* \tTheta_j &=& \pr_j \circ (\Psi_1 \circ P(\fra))\\
&=& \pr_j \circ (\Psi_{m+1} + b_{m-1}\Psi_{m}+ \cdots + q^h \Psi_1 )\\
&=& \pr_j \circ (\Psi_{m+1} + b_{m-1}\Psi_{m}+ \cdots b_1\Psi_2) 
+ \pr_j \circ q^h \Psi_1. \\
\eeqar
By Proposition \ref{diff} (1), note that 
$$
\fra^* (\iota^* \tTheta) = \pr_j \circ (\Psi_{m+1} + b_{m-1}\Psi_{m}+ 
\cdots b_1\Psi_2).
$$

Therefore we have 
\begin{align}
\label{phi-up2}
\Delta^* \tTheta_j := (\phi \circ \iota - \iota \circ \fra)^*\tTheta_j =
 q^h e_j. \Psi_1
\end{align}
where $e_j$ is the $j$-th row of the $g\times g$ identity matrix 
$\mathbf{1}_{g}$.
On the other hand by (\ref{phiup}), for all $j$ there exists 
$c_{1j},\dots, c_{gj} \in R$ such that
\begin{align}
\label{phi-up3}
\Delta^* \tTheta_j &= (\phi(c_{1j}) \gamma_1 + \cdots + 
\phi(c_{gj}) \gamma_g). \Psi_1
\end{align}
Hence comparing equations (\ref{phi-up2}) and (\ref{phi-up3}) for all $j$ we obtain that 
\begin{align}
\label{phi-up4}
C. \Gamma =  q^h \mathbf{1}_g,
\end{align}
where $C= (\phi(c_{ij}))_{ij} \in \Mat_{g\times g}(R)$. Hence $\Gamma$ is invertible
over $K$ and we are done.
\end{proof}

\begin{theorem}
\label{isoupsemi}
Let $G$ be semi-abelian $\pi$-formal group scheme over $S$ of relative 
dimension $g$. 
Then the  $R$-linear map $\Upsilon: \bXp(G) \map (\Lie G)^*$ is an 
injective morphism of $R$-modules with $\pi$-power torsion cokernel.

In particular, $\Upsilon: \bXp(G)_K \map (\Lie G)^*_K$ is an isomorphism
of $K$-vector spaces.
\end{theorem}
\cblue
\begin{proof}
It is sufficient to show that $\Upsilon$ is an isomorphism of 
$K$-vector spaces. By \eqref{Upsilonmatrix}, it is then enough to show 
that $\Gamma$ is an invertible matrix over $K$.
Since $G$ is a semi-abelian scheme of dimension $g$ over $S$, by Proposition
\ref{Frob-char}, the absolute $q$-power Frobenius $F$ on $\ov{G}$ satisfies
the polynomial $P(F)$  in the proposition. 
Hence the result follows from Theorem \ref{isoupsilon}.
\end{proof}

\section{Non-degeneration of the filtered $F$-isocrystals}
\label{non-deg}
In this section, as a consequence of Theorem \ref{isoupsilon}, we will show 
that if $G$ is a commutative smooth $\pi$-formal scheme over $S$, then
$\bH_\d(G)_K$ with the  operator $\fra^*$ is a non-degenerate filtered $F$-isocrystal,
that is, $\fra^*$ is bijective map.

Let $\mathbb{B}= \{\Theta_1,\dots, \Theta_g\}$ be a basis of primitive characters 
as in equation (\ref{Xprimiso}), that is
$$
\bXp(G) \simeq R\langle \Theta_1,\dots , \Theta_g\rangle.
$$ 
where $o_i := \ord{\Theta_i} \leq \mup$ for all $i=1,\dots , g$ and $\mup$ is the upper 
splitting number defined in Section \ref{pre}. For any $\lam \in R$, we will denote 
$\lam^\phi := \phi(\lam)$. Again, we would like to remind the reader of our remark 
at the beginning of Section \ref{SUpsilon} regarding keeping track of the Frobenius twists 
of the coefficients belonging to $R$ even though $\phi$ acts by identity on it.

For any delta character $\Theta$ let
$$
S_n(\Theta):= \{\phi^{*(n-i)}\Theta \mid i = \ord{\Theta},\ord{\Theta}+1, 
\dots , n\}
$$
for all $n \geq \ord{\Theta}$. Recall from page $418$ of \cite{BS_b}
that for all $n\geq \mup$ we have
the set $S_n(\Theta_1),\dots S_n(\Theta_g)$ is $K$-linearly independent and 
freely generates $\bX_n(G)_K$ as a $K$-module.

Define the delta characters $\tTheta_i:= \phi^{*(n-o_i)}\Theta_i$ for all 
$n\geq \mup$ and $i=1,\dots, g$.
Hence the quotient $K$-module $\bX_n(G)_K/u^*\bX_{n-1}(G)_K$ is 
generated by the set $\{\tTheta_1,\dots ,\tTheta_g\}$ as a $K$-module.

Suppose for all $i=1,\dots, g$ we have 
\begin{align}
\iota^* \tTheta_i = 
\lam_{i{\mup}}\Psi_{\mup} - \lam_{i({\mup} -1)}\Psi_{{\mup} -1} 
- \cdots - \lam_{i1} \Psi_1
\end{align}
where $\lam_{ij} \in \Mat_{1\times g}(R)$ for all $j = 1,\dots \mup$.
Then we have 
\begin{align}
\fra^* \iota^* \tTheta_i = 
\lam_{i{\mup}}^\phi\Psi_{{\mup}+1} - \lam_{i({\mup} -1)}^\phi\Psi_{{\mup}} 
- \cdots - \lam_{i1}^\phi \Psi_2.
\end{align}
Hence from Proposition $6.2$ in \cite{BS_b} we have some $\gamma_i \in
\Mat_{1\times g}(R)$ such that
\begin{align}
\label{phimu}
%\nonumber
\iota^*\phi^* \tTheta_i &= \fra^*\iota^* \tTheta_i + \gamma_i \Psi_1\\
&= 
\lam_{i{\mup}}^\phi\Psi_{{\mup}+1} - \lam_{i({\mup} -1)}^\phi\Psi_{{\mup}} 
- \cdots - \lam_{i1}^\phi \Psi_2 + \gamma_i\Psi_1.
\end{align}
For all $j=1,\dots ,\mup$ define the $g \times g$ matrices 
$$\Lambda_{j}:= \left(\begin{matrix}
\lam_{1j}^\phi \\ \vdots \\ \lam_{gj}^\phi \\
\end{matrix}
\right).$$
Therefore the map
$$
[\iota^*\phi^*]:\bX_n(G)_K/u^*\bX_{n-1}(G)_K {\longrightarrow}
\Hom(N^{n+1}G,\hG)_K/u^*\Hom(N^nG,\hG)_K$$
is given by
$[\iota^*\phi^*] \tTheta_i = \lam^\phi_{i\mup} \Psi_{{\mup}+1}$ for all 
$i=1,\dots, g$. 

Hence the map $[\iota^*\phi^*]$ in terms of the basis
$\{\tTheta_1,\dots,\tTheta_g\}$ of $\bX_{\mup}(G)/u^*\bX_{{\mup}-1}(G)$
and basis $\{\Psi_{\mup+1}\}:= \{\Psi_{(\mup+1)1},\dots , \Psi_{(\mup+1)g}\}$ of 
$\Hom(N^{\mup +1}G,\hG)/u^*\Hom(N^{\mup}G,\hG)$ is given by
\begin{align}
[\iota^*\phi^*] = \Lambda_{\mup}
\end{align}

Recall the following proposition from the proof of Proposition $8.1$ in 
\cite{BS_b}
\begin{proposition}
\label{Lambdamu}
For all $n \geq \mup$, we have
$$
\bX_n(G)_K/u^*\bX_{n-1}(G)_K \stk{[\iota^*  \phi^*]}{\longrightarrow}
\Hom(N^{n+1}G,\hG)_K/u^*\Hom(N^nG,\hG)_K
$$
is an isomorphism. 

In particular, for $n=\mup$, we have 
$$
\bX_{\mup} (G)_K/u^*\bX_{\mup-1}(G)_K \stk{[\iota^*  \phi^*]}{\longrightarrow}
\Hom(N^{{\mup} +1}G,\hG)_K/u^*\Hom(N^\mup G,\hG)_K
$$
is an isomorphism and hence $\Lambda_{\mup}$ is an invertible matrix over $K$.
\end{proposition}

\subsection{The operator $\fra^*$}
Let us review the description of the operator $\fra^*$ on $\bH_\d(G)_K$.
Recall that 
$$
\bH_\d(G)_K := \varinjlim_{n} 
\frac{\Hom(N^nG ,\hG)_K}{\iota^* \phi^* \bX_{n-1}(G)_K}
$$
and $\fra^*$ on $\bH_\d(G)_K$ is induced from the pullback on the delta 
characters of $N^nG$ given by $\fra^* \Psi_i = \Psi_{i+1}$ for all $i \geq 1$.
Since $\fra^*(\iota^* \phi^* \bX_{n-1}(G)) \subset \iota^* \phi^* \bX_n(G)$
for all $n\geq 1$, $\fra^*$ induces a semilinear operator on $\bH_\d(G)_K$.

By equation (\ref{phimu}) and Proposition \ref{Lambdamu}, we have 
\begin{align}
\fra^* \Psi_{\mup} = \Psi_{\mup+1} \equiv
\Lambda_{\mup}^{-1}\left( 
\Lambda_{\mup-1}\left(\Psi_{\mup} + \cdots + \Lambda_1 \Psi_2 - \Gamma \Psi_1\right)\right) \bmod
\iota^*\phi^* \bX_{\mup-1}(G)_K
.\nonumber
\end{align}

By Proposition $8.1$ of \cite{BS_b} we have for all $n \geq \mup$ 
$$
\bH_\d(G)_K \simeq \bH_n(G)_K \simeq \bH_{\mup}(G)_K.
$$
where $$\bH_{\mup}(G)_K ={\Hom(N^{\mup}G,\hG)}/{\iota^*\phi^* \bX_{\mup-1}(G)}
\simeq K\langle\Psi_1,\dots , \Psi_{\mup}\rangle/
\iota^*\phi^* \bX_{\mup -1}(G),$$
since $\Hom(N^{\mup}G,\hG) \simeq K\langle \Psi_1, \dots ,\Psi_{\mup}\rangle$.

Consider the following operator (still denoted as $\fra^*$) on 
$\Hom(N^{\mup}G,\hG)_K \simeq K\langle \Psi_1,\dots , \Psi_{\mup}\rangle$ 
given by
\begin{align}
\fra^*(\Psi_i) = \Psi_{i+1} \mb{ for all } i= 1, \dots, \mup -1, \mb{ and }
\nonumber \\
\fra^* \Psi_{\mup} = \Lambda_{\mup}^{-1} \left({\Lambda_{\mup-1}} 
\Psi_{\mup} + \cdots + \Lambda_1 \Psi_2 - \Gamma \Psi_1\right).\nonumber
\end{align}

Note that the above map descends to $\fra^*$ on the quotient $\bH_\d(G)_K$.

Then the $(g\mup  \times g\mup)$-matrix associated to $\fra^*$ with 
respect to the basis $\{\Psi_1, \dots, \Psi_{\mup}\}$ is given by 
\begin{align}
[\fra^*] = 
\left(
\begin{matrix}
[{\bf{0}}]_g & [{\bf{0}}]_g & \cdots & & - \Lambda_{\mup}^{-1}\Gamma \\
[{\bf{1}}]_g & [{\bf{0}}]_g &  & & \Lambda_{\mup}^{-1}\Lambda_1 \\
[{\bf{0}}]_g & [{\bf{1}}]_g &  & &   \\
\vdots & & \ddots & & \vdots \\
[{\bf{0}}]_g & [{\bf{0}}]_g & \cdots  &[\bf{1}]_g & 
\Lambda_{\mup}^{-1}\Lambda_{\mup-1} \\
\end{matrix}
\right)
\end{align}
where $[{\bf{0}}]_g$ and $[{\bf{1}}]_g$ are the $g \times g$ zero and the identity 
matrices respectively. 

\begin{theorem}
\label{bijfra}

If $G$ is a semi-abelian $\pi$-formal scheme of relative dimension $g$ over
$S$, then the operator $\fra^*: \bH_\d(G)_K \map 
\bH_\d(G)_K$ is a bijection 

Hence $\mathrm{\Iso}(\bH_\d(G)_K)$ is a non-degenerate filtered isocrystal.
\end{theorem}

\begin{proof}
Since $G$ is a semi-abelian $\pi$-formal scheme over $S$, by Theorem 
\ref{Frob-char} the absolute Frobenius $F$ on $\ov{G}$ satisfies $P(F) = 0$ 
for some $P(t) = t^m + b_{m-1}t^{m-1} + \cdots + b_1 t + q^h \in \bb{Z}[t]$.
By equation (\ref{phi-up4}) in the 
proof of Theorem \ref{isoupsilon} we have $\Gamma = 
(\gamma_1 \cdots \gamma_g)^t$ is an invertible $g\times g$ matrix over $K$.
Hence we have 
$$
\det [\fra^*] = (-1)^{\mup} \det (\Lambda_{\mup}^{-1} \Gamma) \ne 0.
$$
Therefore the semilinear map $\fra^*: \Hom(N^{\mup}G,\hG)_K \map 
\Hom(N^{\mup}G,\hG)_K$ is bijective. Hence this implies that the map $\fra^*$
on the quotient $\bH_\d(G)_K$ is also bijective.
\end{proof}

\section{Comparison of $\bH_\d(A)$ with $\bH^1_{\mathrm{cris}}(A)$ for 
elliptic curves }
\label{comparison}

Let $G$ be a $\pi$-formal group scheme over $S$. Consider the module of 
endomorphisms of $G$ over $S$ denoted by $\End (G)$. Given an element 
$\beta \in \End (G)$, for all $n$ we have an induced canonical 
endomorphism (still denoted by $\beta$) of $N^nG$ given by the dotted arrow in 
the following diagram:
\begin{align}
\label{kernelend}
\xymatrix{
0 \ar[r] & N^nG \ar@{.>}[d]^-\beta \ar[r]^\iota & J^nG \ar[d]^\beta \ar[r]^u & 
G \ar[d]^\beta \ar[r] & 0\\
0 \ar[r] & N^nG \ar[r]^\iota & J^nG \ar[r]^u & G \ar[r] & 0.\\
}
\end{align}
Therefore we obtain a map $\End(G) \map \End(N^nG)$.
Composing the above map we obtain a representation of $\End(G)$ as follows
\begin{align}
\label{repn}
\rho: \End(G) \map \End(N^nG) \map \End(\Hom(N^nG,\hG))
\end{align}
given by 
$$
\rho(\beta)(\Psi) := \Psi \circ \beta
$$
for all $\beta \in \End(G)$ and $\Psi \in \End(\Hom(N^nG,\hG))$.

In the case when $G$ is a smooth $\pi$-formal group scheme of dimension
$g$ over $S$, by \cite{BS_b} we have 
$$
\End(N^nG,\hG) \simeq R^{ng}.
$$
The equation (\ref{repn}) then becomes
\begin{align}
\label{repn2}
\rho : \End(G) \longrightarrow \End(\Hom(N^nG,\hG)) \simeq 
\Mat_{ng\times ng}(R).
\end{align}

\subsection{Elliptic curve $A$ with Frobenius lift}
Let $A$ be an elliptic curve with a Frobenius lift $\clift \in \End(A)$ on it.
Then $\clift$ satisfies the characteristic polynomial of the Frobenius
$$
t^2 - a_q(\ov{A}) t + q = 0
$$
where $\ov{A} = A \otimes k$. Then the equation (\ref{repn2}) for $n=1$ 
becomes
\begin{align}
\rho : \End(A) \map \End(\Hom(N^1A,\hG)) \simeq R.
\end{align}
In other words, for all $\beta \in \End(A)$ and $\Psi \in \Hom(N^1A,\hG)$
we have $\Psi \circ \beta = \rho(\beta) \Psi$, that is, the following diagram
is commutative
$$\xymatrix{
N^1A \ar[d]_\Psi  \ar[r]^\beta & N^1A \ar[d]^\Psi\\
\hG \ar[r]^{\rho(\beta)} & \hG.
}$$

Also note that for an elliptic curve $A$ over $S$, the Frobenius lift 
endomorphism $\beta \in \End(A)$ induces an $K$-endomorphism $\Fcris: 
H^0(A,\Omega_A)_K \map H^0(A,\Omega_A)_K$ given by
$\Fcris(v) = (D\beta) v$ where $D\beta$ is the derivative map induced by $\beta$
and the identity section for all $v \in H^0(A,\Omega_A)_K$. This makes 
$(H^0(A,\Omega_A)_K,\Fcris)$ into an $F$-isocrystal. 

In the case when $\Psi_1 \in \Hom(N^1A,\hG)$ is as in (\ref{Psi1}), by
(\ref{phider}) we have 
\begin{align}
\label{Psi1beta}
\xymatrix{
N^1A \ar[d]_{\Psi_1}  \ar[r]^\beta & N^1A \ar[d]^{\Psi_1}\\
\hG \ar[r]^{D\beta} & \hG.
}
\end{align}
We now recall some of the results from Section $9$ of \cite{BS_b}.
Let $A$ be an elliptic curve over $S$. Then we have the following two cases:
\begin{enumerate}
\item When $A$ is not a Frobenius lift then, $\bX_\infty(A)$ is generated 
by $\Theta_2 \in \bX_2(A) \backslash \bX_1(A)$ as an $R\{\d \}$-module. By
Proposition \ref{diff} we have
\begin{align}
\label{no-CL}
\iota^* \Theta_2 &= \Psi_2 - \lam \Psi_1, \mb{ and } \\
\iota^* \phi^* \Theta_2 &= \Psi_3 - \phi(\lam) \Psi_2+ \gamma \Psi_1. \nonumber
\end{align}
for some $\lam, \gamma \in R$.

\item When $A$ has a Frobenius lift then, $\bX_\infty(A)$ is generated by
$\Theta_1 \in \bX_1(A)$ as an $R\{\d\}$-module. By Proposition \ref{diff}
we have
\begin{align}
\label{with-CL}
\iota^* \Theta_1 &= \Psi_1, \mb{ and } \\
\iota^* \phi^* \Theta_1 & = \Psi_2 + \gamma \Psi_1. \nonumber
\end{align}
for some $\gamma \in R$.
\end{enumerate}

Recall Theorem $9.7$ of \cite{BS_b}:

\begin{theorem}
\label{fra-map}
Let $A$ be an elliptic curve over $S$.
\begin{enumerate}
\item If $A$ is not a Frobenius lift, then 
$$\bH_{\d}(A) \simeq R\langle \Psi_1, \Psi_2 \rangle, ~ \bXp(A) \simeq 
R\langle \Theta_2 \rangle.$$ 
The semilinear operator $\fra^*$ acts as 
$$\fra^*(\Psi_1) = \gamma \Psi_2,~ \fra^*(\Psi_2) = \phi(\lam)\Psi_2 -
\gamma \Psi_1.$$

\item If $A$ is a Frobenius lift, then 
$$\bH_{\d}(A) \simeq R\langle \Psi_1 \rangle, ~ \bXp(A) \simeq 
R\langle \Theta_1 \rangle.$$ 
The semilinear operator $\fra^*$ acts as 
$$\fra^*(\Psi_1) = -\gamma \Psi_1.$$
\end{enumerate}
\end{theorem}

Now we fix the basis $\bb{B}= \{\Psi_1,\Psi_2\}$ when $A$ is not a Frobenius
lift and $\bb{B} = \{\Psi_1\}$ when $A$ is a Frobenius lift. Let 
$[\fra^*]_{\bb{B}}$ denote the matrix representation of $\fra^*$ with respect
to the fixed $\bb{B}$. Let $\mathrm{char}~{[\fra^*}]_{\bb{B}}
 \in R[t]$ denote the characteristic polynomial of the matrix 
$[\fra^*]_{\bb{B}}$. 

\begin{corollary}
\label{char-poln}
Let $A$ be an elliptic curve over $S$. 
\begin{enumerate}
\item If $A$ is not a Frobenius lift then $\mathrm{char }~{[\fra^*]}_{\bb{B}} = 
t^2 - \phi(\lam)t + \gamma$

\item If $A$ has a Frobenius lift then $\mathrm{char}~{[\fra^*]}_{\bb{B}} = 
t + \gamma$.
\end{enumerate}
\end{corollary}
\begin{proof}
The result follows directly from Theorem \ref{fra-map}.
\end{proof}

\begin{lemma}
\label{theta-lemma}
Let $A/R$ be a non-CL elliptic curve over $S$ and  $P(t) = t^2 - a_q(A)t + q$ 
be the characteristic polynomial satisfied by the absolute  Frobenius 
$\ov{F}$ in $\End(\ov{A})$. Then 
$$
\iota^* \phi^* \Theta_2 = \Psi_3 - a_q(A) \Psi_2 + q \Psi_1.
$$
In particular, we have $\lam = a_q(A)$ and $\gamma = q$ and 
$$ \mathrm{char}~{[\fra^*]}_{\bb{B}} = P(t) = t^2 - a_q(A)t + q. $$
\end{lemma}

\begin{proof}
For elliptic curve $A$ the diagram (\ref{frob-diag}) turns out to be as follows:
$$
\xymatrix{ 
 N^{3}A \ar[r]^-{\phi\circ \iota} \ar[d]_-{P(\mfrak{f})} & J^{2}A \ar[r] 
\ar@{.>}[dl]_{\mfg} \ar[d]_-{P(\mfrak{\phi})} & \overline{A} 
\ar[d]_-{P(\overline{F})} \\
 N^1A\ar[d]_{\Psi_{1}} 
\ar[r]^-{\phi \circ \iota} & A \ar[r]^-h & \overline{A} \\
\hG. & & 
}
$$
Hence we have 
\beqar
\iota^* \phi^* \Theta_2  &=& \Psi_1 \circ P(\fra) \\
&=& \Psi_1 \circ (\fra^2 - a_q(A) \fra +q)\\
&=& \Psi_3 -a_q(A) \Psi_2 + q \Psi_1.
\eeqar

Comparing the above with (\ref{no-CL}) we obtain $\phi(\lam) = a_q(A)$ 
which implies $\lam = a_q(A)$ (since $a_q(A) \in \Z_p$) and $\gamma = q$
and the equality of the characteristic polynomials follows from the above and
Corollary \ref{char-poln} $(1)$.
\end{proof}

\begin{lemma}
\label{gamma-rho}
Let $A$ be an elliptic curve over $S$ which is a Frobenius lift and $P(t) =
t^2 - a_q(A)t + q$ be the characteristic polynomial satisfied by the absolute
Frobenius $\ov{F}$ in $\End(\ov{A})$. Then
$$
\iota^*\phi^* \Theta_1 = \Psi_2 - (D\beta) \Psi_1
$$
\end{lemma}
In particular, we have $\gamma = -(D\beta)$.

\begin{proof}
For an elliptic curve $A$ we have the following diagram as follows:
$$
\xymatrix{ 
 N^{2}A \ar[r]^-{\phi\circ \iota} \ar[d]_-{\fra - \beta} & J^{1}A \ar[r] 
\ar@{.>}[dl]_{\mfg} \ar[d]^-{\phi - \beta} & \overline{A} 
\ar[d]^-{0 } \\
 N^1A\ar[d]_{\Psi_{1}} 
\ar[r]^-{\phi \circ \iota} & A \ar[r]^-h & \overline{A} \\
\hG. & & 
}
$$
Hence we have 
\beqar
\iota^* \phi^* \Theta_1  &=& \Psi_1 \circ (\fra - \beta) \\
&=& \Psi_1 \circ \fra - \Psi_1 \circ \beta \\
&=& \Psi_2 - (D\beta)\Psi_1, \mbox{ by (\ref{Psi1beta})}.
\eeqar

Comparing the above with (\ref{with-CL}) we obtain $\gamma = -(D\beta)$.
\end{proof}

\cblue

Let $\mr{\Iso}(H) = (H,F,(H \supset V \supset \{0\}))$ be a filtered $F$-isocrystal over $K$ 
where $H$ is a two dimensional vector space over $K$, $F:V \map V$ is a 
semilinear (in fact, in this case $F$ is linear since $\phi$ is the lift of $q$-power Frobenius, which is identity on $K$)
operator which is a bijection and $V$ is a one-dimensional $K$-subspace 
of $H$. Let $p_F(t) \in K[t]$ be the degree two
characteristic polynomial of $F$.

\begin{proposition}
\label{filiso}
Let $\mr{\Iso}(H)=(H,F,(H \supset V \supset \{0\}))$ and 
$\mr{\Iso}(H')=
\newline
(H',F',(H' \supset V' \supset \{0\}))$ 
be filtered $F$-isocrystals over $K$ such that 
\begin{enumerate}
\item $\dim_{K} H = \dim_{K} H' = 2$
\item $F(V) \ne V$ and $F'(V') \ne V'$ and
\item $p_F(t) = p_{F'}(t)=: p(t)$. 

\end{enumerate}
Then $\mr{\Iso}(H) \simeq \mr{\Iso}(H')$ in the category of filtered $F$-isocrystals
over $K$.
\end{proposition}

\begin{proof}

Let $p(t) = t^2 - at -b$ for some $a,b \in K$ and choose any non-zero vector
 $v \in V$. Since $F(V) \ne V$, the set $\{v, F(v)\}$ forms a $K$-basis of 
$H$. Then we have 
$$
F^{\circ 2}(v) = a F(v) + b v.
$$
Similarly, for any non-zero vector $w \in V'$ the set $\{w,F'(w)\}$ is a 
$K$-basis for $H'$ and we have
$$
F'^{\circ 2}(w) = a F'(w) + b w.
$$
Define the $K$-linear map $\Phi: H \map H'$ given by $\Phi(v):= w$ and 
$\Phi(F(v)) := F'(w)$. Then $\Phi$ is an isomorphism of $K$-vector spaces
that satisfies
$$
\xymatrix{
H \ar[d]_F \ar[r]^\Phi & H'\ar[d]^{F'}\\
H \ar[r]^\Phi & H'
}
$$
such that $\Phi(V) = V'$. Hence $\Phi:\mathrm{\Iso}(H) \map \mathrm{\Iso}(H')$
is the required isomorphism of filtered $F$-isocrystals and this completes our
proof.

\end{proof}

\begin{proof}[{\bf Proof of Theorem \ref{Iso-crys-11}}]
We have the following two cases depending on whether $A$ admits a lift of Frobenius or not. 

(1) \textbf{\underline{Non-CL case}:} 
Consider the filtered $F$-isocrystal 
$$
\mathrm{\Iso}(\bH_\d(A)_{K}) = (\bH_\d(A)_{K}, \fra^*, 
\bH_\d(A)_{K}^\bullet)
$$
where $\bH_\d(A)_{K}^\bullet$ is the filtration given by 
$\bH_\d(A)_{K} \supset \bXp(A)_{K} \supset \{0\}$.
Since $A$ is non-CL, we have $\fra^*(\bXp(A)_{K}) \ne \bXp(A)_{K}$ and
the characteristic polynomial of $\fra^*$ is $p_{\fra^*}(t) = t^2 - a_q(A) t + q$.

On the other hand, consider the filtered $F$-isocrystal of the first crystalline
cohomology $\mathrm{\Iso}(\Hcr(A)_{K})= (\Hcr(A)_{K}, \Fc,
\Hcr(A)^\bullet_K)$ where $\Hcr(A)^\bullet_K$ is the Hodge filtration $\Hcr(A)_K 
\supset H^0(A,\Omega_A)_K \supset \{0\}$ and $\Fc$ is the crystalline Frobenius
operator on $\Hcr(A)_K$.
Since $A$ is a non-CL elliptic curve over $R$, by Theorem $3.15$ of 
\cite{BO},
$\Fc (H^0(A,\Omega_A)_{K}) \ne H^0(A,\Omega_A)_{K}$ and the characteristic
polynomial of $\Fc$ is $p_{\Fc}(t) = t^2 - a_q(A) t + q$. 

Hence by Proposition \ref{filiso} applied to $\mathrm{\Iso}(\bH_\d(A)_{K})$
and $\mathrm{\Iso}(\Hcr(A)_{K})$ we obtain our required isomorphism of 
filtered $F$-isocrystals.

\vspace{.2cm}

\noindent (2) \textbf{\underline{CL case}:} 
If $\beta:A  \map A$ is the Frobenius lift map, then $D\beta$ is the induced
map on the Lie algebra $T_0A$. By (\ref{Psi1beta}) and
Lemma \ref{gamma-rho}, it follows that $-\gamma = D\beta$. Hence we have
$$
\Fc(\omega) = - \gamma \omega
$$
for all $\omega  \in H^0(A,\Omega_A)$.

On the other hand, by Theorem \ref{isoupsilon} and Theorem \ref{fra-map}(2) we 
have $\Hd(A)_K =\bXp(A)_K \stk{\Upsilon}{\simeq} H^0(A,\Omega_A)_K$ and
$\fra^*(v) = -\gamma v$ for all $v \in \Hd(A)_K$ where $\Upsilon$ is as in
Theorem \ref{isoupsilon}. 
Hence $\Upsilon$ gives us the required isomorphism of $F$-isocrystals and 
this completes the proof.
\end{proof}

\vspace{.5cm}
\noindent {\bf Acknowledgements.}
The authors would  like to thank James Borger and Netan Dogra for insightful comments. The second author was partly supported by the Vikram Sarabhai Research Fellowship at IIT Gandhinagar and the Postdoctoral Fellowship at IISER Mohali. 
The authors are grateful to the anonymous referee for the careful reading of the manuscript and for the insightful comments and suggestions, which have significantly improved the exposition and clarity of the paper.
The third author is also thankful to Kiran Kedlaya for his perceptive remarks during a helpful discussion on this manuscript.
He would also like to thank the Lodha Mathematical Sciences Institute, Mumbai for its kind hospitality and support.

\color{black}

\footnotesize{

}

\end{document}